\documentclass[a4paper]{amsart}

\usepackage[english]{babel}
\usepackage[utf8x]{inputenc}
\usepackage[T1]{fontenc}
\usepackage{bm}

\usepackage[a4paper,top=3cm,bottom=2cm,left=3cm,right=3cm,marginparwidth=1.75cm]{geometry}

\usepackage{amsmath, latexsym, amssymb, amsthm}
\usepackage{graphicx}
\usepackage{enumerate}
\usepackage[colorinlistoftodos]{todonotes}
\usepackage[colorlinks=true, allcolors=blue]{hyperref}
\newtheorem{thm}{Theorem}[section]
\newtheorem{conj}[thm]{Conjecture}
\newtheorem{prop}[thm]{Proposition}
\newtheorem{lem}[thm]{Lemma}
\theoremstyle{plain}
\newtheorem{defn}[thm]{Definition}

\usepackage{comment}
\usepackage{tikz}
\usepackage[version=4]{mhchem}          

\newcommand{\op}[1]{\operatorname{#1}}
\newcommand{\ol}[1]{\operatorname{\sf #1}}
\newcommand{\inn}[3]{\langle {#1}_{#2}, {#1}_{#3} \rangle}

\newcommand{\fivegraphs}{%
\begin{figure}
\centering
\begin{tikzpicture}[scale=0.7]
\draw (-2,0) -- (2,0);
\draw (0,0) -- (0,2);
\foreach \x in {-2,-1,0,1,2}
  \draw[fill] (\x,0) circle (2pt);
\draw[fill] (0,1) circle (2pt);
\draw[fill] (0,2) circle (2pt);
\node at (0,-1) {Type III: $\widetilde{E}_6$};
\begin{scope}[xshift=3cm]
\draw (0,0) -- (6,0);
\draw (3,0) -- (3,1);
\foreach \x in {0,1,...,6}
  \draw[fill] (\x, 0) circle (2pt);
\draw[fill] (3,1) circle (2pt);
\node at (3,-1) {Type IV: $\widetilde{E}_7$};
\end{scope}
\begin{scope}[xshift=10cm]
\draw (0,0) -- (7,0);
\draw (2,0) -- (2,1);
\foreach \x in {0,1,...,7}
  \draw[fill] (\x, 0) circle (2pt);
\draw[fill] (2,1) circle (2pt);
\node at (3.5,-1) {Type V: $\widetilde{E}_8$};
\end{scope}
\begin{scope}[yshift=6cm]
\draw (0,0) arc [start angle=-180, end angle=90, radius=2] (2,2);
\draw [dashed] (2,2) arc [start angle=90, end angle=180, radius=2] (0,0);
\foreach \x in {-180,-135,...,90}
  \draw[fill] (2,0) + (\x:2) circle (2pt);
\node at (2,-3) {Type I: $\widetilde{A}_\ell$};
\begin{scope}[xshift=3cm]
\draw (4,1) -- (5,0) -- (4,-1);
\draw (5,0) -- (7,0);
\draw [dashed] (7,0) -- (9,0);
\draw (9,0) -- (11,0);
\draw (12,1) -- (11,0) -- (12,-1);
\foreach \x in {5,6,7,9,10,11}
\draw [fill] (\x, 0) circle (2pt);
\draw [fill] (4,  1) circle (2pt);
\draw [fill] (4, -1) circle (2pt);
\draw [fill] (12, 1) circle (2pt);
\draw [fill] (12,-1) circle (2pt);
\node at (8,-3) {Type II: $\widetilde{D}_\ell$};
\end{scope}
\end{scope}
\end{tikzpicture}
\caption{The five types of connected graphs with maximum eigenvalue $2$, which are also the Dynkin diagrams of affine reflection groups~\cite{humphreys1992reflection}}
\label{fig:graphtypes}
\end{figure}
}

\title{Equiangular lines and the Lemmens-Seidel conjecture}
\author[Yen-chi.~R.~Lin]{Yen-chi Roger Lin}
\address{Department of Mathematics, National Taiwan Normal University, Taipei, Taiwan 11677}
\email{yclinpa@gmail.com}
\author{Wei-Hsuan Yu}
\address{Department of Mathematics, National Central University, Chungli, Taoyuan, Taiwan 32001}
\email{u690604@gmail.com}
\date{\today}
\keywords{Equiangular set, pillar methods, Lemmens-Seidel}
\subjclass[2010]{52C35}

\begin{document}
\begin{abstract}
In this paper, claims by Lemmens and Seidel in 1973 about equiangular sets of lines with angle $1/5$ are proved by carefully analyzing pillar decomposition, with the aid of the uniqueness of two-graphs on $276$ vertices. 
The Neumann Theorem is generalized in the sense that if there are more than $2r-2$ equiangular lines in $\mathbb{R}^r$, then the angle is quite restricted. 
Together with techniques on finding saturated equiangular sets, we determine the maximum size of equiangular sets ``exactly'' in an $r$-dimensional Euclidean space for $r = 8$, $9$, and $10$.
\end{abstract}

\maketitle


\section{Introduction}
\label{sec:intro}
A set of lines in Euclidean space is called \emph{equiangular} if any pair of lines forms the same angle. 
For examples, the four diagonal lines of a cube are equiangular in $\mathbb{R}^3$ with the angle $\arccos (1/3)$, 
and the six diagonal lines of an icosahedron form $6$ equiangular lines with angle $\arccos (1/\sqrt 5)$. 
The structure of methane \ce{CH4} also contains equiangular lines: 
carbon-hydrogen chemical bounds form the same angle (about $109.5$ degrees). 
Equiangular lines in real and complex spaces are related to many beautiful mathematical topics and even quantum physics, such as SIC-POVM~\cite{renes2004symmetric, scott2010symmetric, scott2006tight, zauner2011grundzuge}.
First, equiangular lines in real spaces are equivalent to the notion of \emph{two-graphs} which caught much attention in algebra~\cite{godsil2013}. 
A classical way to construct equiangular lines comes from combinatorial designs. 
For instance, the 90 equiangular lines in $\mathbb{R}^{20}$ and 72 equiangular lines in $\mathbb{R}^{19}$ can be obtained from the Witt design. 
The details can be found in Taylor's thesis in 1971~\cite{taylor1971}. 
The spherical embedding of certain strongly regular graphs can also give arise to equiangular lines ~\cite{cameron2004strongly}; 
the maximum size of equiangular lines in $\mathbb{R}^{23}$ is $276$ which can be constructed from the strongly regular graphs with parameters $(276,135,78,54)$. 
Such configuration is the solution to the energy minimizing problems~\cite{saff1997distributing}, also known as the \emph{Thomson Problem}. 
The Thomson problem, posed by the physicist J.~J.~Thomson in 1904~\cite{thomson1904xxiv}, is to determine the minimum electrostatic potential energy configuration of $N$ electrons constrained to the surface of a unit sphere that repel each other with a force given by Coulomb's law. 
The configuration of several maximum equiangular lines would give arise to the minimizer of a large class of energy minimizing problems called the \emph{universal optimal codes}~\cite{cohn2007universally}. 
Furthermore, if we have $\frac{r(r+1)}{2}$ equiangular lines in $\mathbb{R}^r$ (which is known as the \emph{Gerzon bounds}~\cite{lemmens1973}), then they will offer the construction of tight spherical $5$-designs~\cite{delsarte1977spherical} which are also universal optimal codes. 
So far, only when $r=2,3,7$, and $23$ can the Gerzon bounds be achieved.
The special sets of equiangular lines, called \emph{equiangular tight frames} (ETFs) refer to the optimal line packing problems~\cite{mixon2018short}. 
ETFs achieve the classical Welch bounds~\cite{welch1974lower} which are the lower bounds for maximum absolute value of inner product values between distinct points on unit sphere, i.e. if we have $M$ points $\{x_i\}_{i=1}^M$ on the unit sphere in $\mathbb{R}^{r}$, then 
\begin{equation*}
\max_{i \neq j} |\langle x_i,x_j \rangle| \geq \sqrt{\frac{M-r}{r(M-1)}}.
\end{equation*}
The study of ETFs has numerous references \cite{fickus2018tremain, strohmer2003, fickus2016equiangular, jasper2014kirkman, barg2015finite, waldron2009construction, strohmer2003}. 

From another point of view, a set of equiangular lines can be regarded as the collection of points on the unit sphere such that distinct points in the set have mutual inner products either 
$\alpha$ or $-\alpha$ for some $\alpha \in [0,1)$.
Below we formally state its definition.
\begin{defn}
\label{defn:equiangular}
We say that a finite set of unit vector $X = \{ x_1, \dots, x_s \}$ in $\mathbb R^r$ is an \emph{equiangular set} if for some $\alpha \in [0,1)$,
\begin{equation}
\label{eq:equi-a}
\langle x_i, x_j \rangle \in \{ -\alpha, \alpha \} \qquad \text{whenever $i \neq j$}.
\end{equation}
\end{defn}
By abuse of language, we will say that a set of vectors which satisfy the condition~(\ref{eq:equi-a}) are equiangular with \emph{angle} $\alpha$, although the actual angle of intersection is $\arccos \alpha$. 
A natural question in this context is: what is the maximum size of equiangular sets in $\mathbb{R}^r$?
We denote by $M(r)$ for this quantity. 
The values of $M(r)$ were extensively studied over the last 70 years.
It is easy to see that $M(2) = 3$ and the maximum construction is realized by the three diagonal lines of a regular hexagon. 
In 1948, Haantjes~\cite{haantjes1948} showed that $M(3) = M(4) = 6$.
In 1966, van~Lint and Seidel~\cite{vanlint1966} showed that $M(5) = 10$, $M(6) = 16$, and $M(7) \geq 28$. 
Currently, there are only 35 known values for $M(r)$ and all of them have that $r \leq 43$.  
To the best of our knowledge, the ranges of $M(r)$ for $2 \leq r \leq 43$ are listed in Table~\ref{tb:smallnd} (see~\cite{azarija2016,barg2014,greaves2016,greaves2018equiangular,greaves2019equiangular,lin2018saturated,yu2015}).

\begin{table}[h]
	\centering
    \caption{Maximum cardinalities of equiangular lines for small dimensions}
    \label{tb:smallnd}
    \begin{tabular}{c|ccccccccc}
   		$r$    & 2 & 3--4 & 5  & 6  & 7--13 & 14     & 15 & 16     & 17 \\ \hline
        $M(r)$ & 3 & 6    & 10 & 16 & 28    & 28--29 & 36 & 40--41 & 48--49 \\
        \hline\hline
        $r$    & 18     & 19     & 20     & 21  & 22  & 23--41 & 42       & 43 \\ \hline
        $M(r)$ & 56--60 & 72--75 & 90--95 & 126 & 176 & 276    & 276--288 & 344
    \end{tabular}
\end{table}

Note that for the dimensions $r=14,16,17,18,19,20$, determining the exact values of $M(r)$ is still an open problem; though we know that the current well-known maximum constructions of equiangular lines are saturated \cite{lin2018saturated}, i.e.\ the current maximum constructions of equiangular lines cannot be added any more line while keeping equiangular. 
The estimation of upper bounds for equiangular lines can be considered from several different methods. 
The bounds could be achieved by semidefinite programming method~\cite{barg2014, okuda2016new, glazyrin2018upper}, the analysis of eigenvalues of the Seidel matrices~\cite{greaves2016, greaves2018equiangular, greaves2018equiangular}, polynomial methods~\cite{glazyrin2018upper}, Ramsey theory for asymptotic bounds~\cite{balla2018equiangular}, forbidden subgraphs for graphs of bounded spectral radius~\cite{jiang2017forbidden, jiang2019equiangular}, and algebraic graphs theory~\cite{godsil2013, neumaier1989graph}. 

The motivation for the study of equiangular lines can also be various. 
For instance, Bannai, Okuda and Tagami~\cite{bannai2015spherical} considered the tight harmonic index 4-designs problems and proved that the existence of tight harmonic index 4-designs is equivalent to the existence of $\frac{(r+1)(r+2)}{6}$ equiangular lines with angle $\sqrt{\frac{3}{r+4}}$ in $\mathbb{R}^r$. 
Later, Okuda-Yu~\cite{okuda2016new} proved such equiangular lines do not exist for all $r > 2$. 
For more information about harmonic index $t$-designs, please see the references~\cite{bannai2015spherical, zhu2017spherical, bannai2018half, bannai2018classification}.   

The main contribution for this paper is that we proved the result which Lemmens-Seidel claimed true in 1973. 
In \cite{lemmens1973}, Lemmens and Seidel claimed that the following conjecture holds when the base size $K = 2, 3, 5$ (for the definition of base size, see Definition~\ref{defn:base-size}):

\begin{conj}[\cite{lemmens1973}, Conjecture~5.8]
\label{conj:LS}
The maximum size of equiangular sets in $\mathbb{R}^r$ for angle $\frac{1}{5}$ is $276$ for $23 \leq r \leq 185$,
and $\lfloor \frac12 (r-5) \rfloor + r + 1$ for $r \geq 185$.
\end{conj}

Although the conjecture was prominent in the study of equiangular lines,
no proof was found in the literature for the cases $K=3,5$.
Following the discussion of pillar methods, we use techniques from linear algebra, linear programming, and the uniqueness of the two-graphs with $276$ vertices to prove the $K=3,5$ cases, and offer a partial solution for $K=4$. 
We also offer better upper bounds for the equiangular sets for some special setting on pillar conditions. 

There is another interesting phenomenon that receives our attention.
It is well known that $M(8) = 28$ (see~Table~\ref{tb:smallnd}),
but those $28$ lines always live in a $7$-dimensional subspace of $\mathbb R^8$ (\cite{glazyrin2018upper}, Theorem 4).
Glazyrin and Yu~\cite{glazyrin2018upper} asks the maximum size of equiangular sets of general \emph{ranks}.
The following theorem essentially states that the angle is restricted when the size of equiangular set is large enough.
\begin{thm}[Neumann, cf.~\cite{lemmens1973}]
\label{thm:neumann}
Let $X$ be an equiangular set with angle $\alpha$ in $\mathbb R^r$.
If $|X| > 2r$, then $\frac{1}{\alpha}$ is an odd integer.
\end{thm}
We first give a generalization of the Neumann theorem (see Theorem~\ref{thm:general Neu}),
then we employ the techniques about saturated equiangular sets in~\cite{lin2018saturated} to determine the maximum size of equiangular sets of ranks $8$, $9$, and $10$.  

The organization of the paper is as follows.
In Section~\ref{sec:prerequisites} we review the basic notations in the study of equiangular sets and recall the pillar decomposition introduced by Lemmens and Seidel~\cite{lemmens1973}.
In Section~\ref{sec:schur} we determine the maximum size of a pillar with orthogonal vectors only. 
In Section~\ref{sec:a5} we provide a proof for the Lemmen-Seidel conjecture when the base size $K = 3$ or $5$, and also give a new upper bound for $K = 4$.
In Section~\ref{sec:max-rank} we discuss the maximum size of equiangular sets of prescribed rank.
We close this paper with some discussions and proposing two conjectures based on our computations.


\section{Prerequisites}
\label{sec:prerequisites}

Throughout this paper, $\hat{x}$ denotes the unit vector in the same direction as a non-zero vector $x$ in an Euclidean space.
We start with some basic definitions for equiangular sets.
Let $X$ be an equiangular set with angle $\alpha$ in $\mathbb R^r$.
There are a few mathematical objects that could be associated to $X$.

\begin{defn}
\label{defn:gramian}
Let $X = \{ x_1, \dots, x_s \} \in \mathbb R^r$ be a finite set of vectors.  The \emph{Gram matrix} of $X$, denoted by $G(X)$ or $G(x_1, \dots, x_s)$, is the matrix of mutual inner products of $x_1$, \dots, $x_s$; that is,
\begin{equation*}
G(X) = X^{\ol{T}} X = \begin{bmatrix}
\langle x_i, x_j \rangle
\end{bmatrix}_{i,j=1}^s
\end{equation*}
\end{defn}

When $X$ is equiangular with angle $\alpha$, then its Gram matrix $G(X)$ is symmetric and positive semidefinite, with entries $1$ along its diagonal and $\pm \alpha$ elsewhere.
The rank of $G(X)$ is the dimension of the span of vectors in $X$; $X$ is linearly independent if and only if $G(X)$ is of full rank (or equivalently, positive definite).

\begin{defn}
\label{defn:seidel-graph}
For an equiangular set $X = \{ x_1, \dots, x_s \}$ with angle $\alpha$, the \emph{Seidel graph} of $X$ is a simple graph $S(X)$ whose vertex set is $X$, and two vertices $x_i$ and $x_j$ of $S(X)$ are adjacent if and only if $\langle x_i, x_j \rangle = -\alpha$.
\end{defn}

Since we are interested in equiangular lines in $\mathbb R^r$, choices need to be made between two unit vectors that span the same line.
However, the choices could affect the signs of their mutual inner products.
If two sets of vectors represent the same set of lines, they are called in the same \emph{switching class}.
This terminology comes from the graph theory: if we \emph{switch} a vertex $v$ in a simple graph, the resulting graph is obtained by removing all edges that are incident to $v$ but adding edges connecting $v$ to all vertices that were not adjacent to $v$.
We also have the freedom to relabel the vertices of the graph.
All these actions lead to the following proposition about the switching equivalence for two Gram matrices.

\begin{prop}[\cite{king2016}, Definition~4]
\label{prop:switching-equivalent}
Two sets of unit vectors $X$, $Y$ in $\mathbb R^r$ are in the same switching class if and only if there are a diagonal $(1,-1)$-matrix $B$ and a permutation matrix $C$ such that
\begin{equation*}
(CB)^{\ol{T}} \cdot G(X) \cdot (CB) = G(Y).
\end{equation*}
We would also say that $G(X)$ is switching equivalent to $G(Y)$, and write $G(X) \simeq G(Y)$.
\end{prop}

As usual, let $I_s$ (resp.\ $J_s$) denote the identity matrix (resp.\ all-one matrix) of size $s \times s$; the subscript $s$ will sometimes be dropped when the size is clear from the context.

\begin{prop}[\cite{lemmens1973}, Section~4]
\label{prop:k-a1}
If there are $k \geq 2$ equiangular vectors $p_1$, \dots, $p_k$ such that
\begin{equation*}
G(p_1, \dots, p_k) \simeq (1 + \alpha) I - \alpha J, \qquad \alpha > 0,
\end{equation*}
then $k \leq \frac{1}{\alpha} + 1$.  Furthermore, if $k < \frac{1}{\alpha} + 1$, then the vectors $p_1$, \dots, $p_k$ are linearly independent; but if $k = \frac{1}{\alpha} + 1$, then the vectors $p_1$, \dots, $p_k$ are linearly dependent.  In fact, if $k = \frac{1}{\alpha} + 1$ and $G(p_1, \dots, p_k) = (1 + \alpha) I - \alpha J$, the vectors $p_1$, \dots, $p_k$ form a $k$-simplex in $\mathbb R^{k-1}$.
\end{prop}

Under a suitable choice of signs, the vectors $\pm p_1$, \dots, $\pm p_k$ from an equiangular set $X$ will form a $k$-clique in its Seidel graph.
Following~\cite{lemmens1973}, we will define two important notions that are associated to an equiangular set $X$
(Definitions~\ref{defn:base-size} and \ref{defn:K-base}).

\begin{defn}[\cite{lemmens1973}]
\label{defn:base-size}
Let $X$ be an equiangular set in $\mathbb R^r$ with angle $\alpha$.
The \emph{base size} of $X$, denoted by $K(X)$, is defined as
\begin{equation*}
K(X) := \max \{ k \in \mathbb N \colon \text{there exist $p_1$, \dots, $p_k$ in $X$ such that $G(p_1, \dots, p_k) \simeq (1 + \alpha) I - \alpha J$} \}.
\end{equation*}
In other words, $K(X)$ is the maximum of the clique numbers of Seidel graphs that are switching equivalent to that of $X$.
\end{defn}

Note that the clique numbers of Seidel graphs in the switching class of $X$ are not constant, therefore we need to take their maximum.
Nevertheless $K(X)$ is always bounded by $\frac{1}{\alpha} + 1$ by Proposition~\ref{prop:k-a1}.
Since we are interested in large equiangular sets, we will assume that $\frac{1}{\alpha}$ is an odd integer, thanks to Theorem~\ref{thm:neumann}.
The following proposition states that the only meaningful range of base size is $2, 3, \dots, \frac{1}{\alpha} + 1$.

\begin{prop}[\cite{king2016}, Proposition~3]
Let $X$ be an equiangular set in $\mathbb R^r$.  
If $|X| \geq 2$, then $K(X) \geq 2$.
\end{prop}

\begin{proof}
If two vertices in the Seidel graph $S(X)$ are independent,
then we switch one of the them to form a $2$-clique.
\end{proof}

\begin{defn}[\cite{lemmens1973}]
\label{defn:K-base}
Let $X$ be an equiangular set with angle $\alpha$ and base size $K$.
A set of $K$ vectors $p_1$, \dots, $p_K$ is called a \emph{$K$-base} of $X$ if $p_1$, \dots, $p_k$ belong to some set which is switching equivalent to $X$, and $G(p_1, \dots, p_K) = (1 + \alpha) I - \alpha J$.
\end{defn}

Let $K$ be the base size of an equiangular set $X$.
We will fix a $K$-base $P = \{ p_1, \dots, p_K \}$ that forms a $K$-clique in the Seidel graph of $X$.
Now we introduce the \emph{pillar decomposition} of $X$ with respect to $P$, following~\cite{lemmens1973}.
(More details can also be found in~\cite{king2016}.)

For each vector $x \in X \setminus P$, there is a $(1,-1)$-vector $\varepsilon(x) \in \mathbb R^K$ such that
\begin{equation*}
	\bigl( \langle x, p_1 \rangle, \dots, \langle x, p_K \rangle \bigr) 
    = \alpha \cdot \varepsilon(x).
\end{equation*}
A vector $x$ in $X$ will be replaced by $-x$ if $\varepsilon(x)$ has more positive entries than $\varepsilon(-x)$, or $\varepsilon(x)$ has the same number of positive entries as $\varepsilon(-x)$ and $\langle x, p_K \rangle = \alpha$; otherwise the vector $x$ stays put.

Let $\Sigma(\varepsilon(x))$ denote the number of positive entries in $\varepsilon(x)$.
A \emph{pillar} (with respect to a $K$-base $P$) containing a vector $x \in X \setminus P$, denoted by $\bar{x}$, is the subset of vectors $x' \in X \setminus P$ such that $\varepsilon(x') = \varepsilon(x)$; $\bar{x}$ is called a $(K,n)$ pillar when $\Sigma(\varepsilon(x)) = n$.
Thus the vectors in $X \setminus P$ are partitioned into several $(K,n)$ pillars for $1 \leq n \leq \lfloor \frac{K}{2} \rfloor$.
The number of different $(K,n)$ pillars is at most $\binom{K}{n}$ when $1 \leq n < \frac{K}{2}$, but is at most $\frac{1}{2} \binom{K}{K/2}$ when $n = \frac{K}{2}$.
However, if $K = \frac{1}{\alpha} + 1$, then $p_1, \dots, p_K$ form a $K$-simplex and $\sum_{i=1}^K p_i = 0$.
Therefore $\varepsilon(x)$ has the same number of positive entries as negative entries, thus only $(K, \frac{K}{2})$ pillars can exist.
The collection of all $(K,n)$ pillars in an equiangular set $X$ will be denoted by $X(K,n)$.


The following fact will be used in many occasions.
\begin{prop}
\label{prop:K1-indep}
Let $X$ be an equiangular set with angle $\alpha$ and base size $K$,
and $P = \{ p_1, \dots, p_K \}$ be a $K$-base.
If two vectors $x, y$ belong to the same $(K,1)$ pillar with respect to $P$,
then $\langle x, y \rangle = \alpha$.
\end{prop}

\begin{proof}
By definition of $x$ and $y$ being in the same $(K,1)$ pillar,
there are $K-1$ vectors in $P$ to which both $x$ and $y$ are adjacent in the Seidel graph $S(X)$ of $X$.
If $x$ and $y$ are also adjacent to each other in $S(X)$,
$x$ and $y$ together with those $K-1$ vectors that they are connected to form a $(K+1)$-clique in $S(X)$,
which contradicts to the definition of the base size $K = K(X)$.
Hence there is no edge connecting $x$ and $y$ in $S(X)$,
which is equivalent of saying that $\langle x, y \rangle = \alpha > 0$.
\end{proof}


\section{Schur decomposition for symmetric positive semidefinite matrices}
\label{sec:schur}

In checking a matrix being positive (semi-)definite, we use the Schur decomposition.

\begin{thm}[Schur decomposition~\cite{boyd2004convex}]
\label{thm:schur-decomp}
Let $M$ be a symmetric real matrix, given by blocks
\begin{equation*}
M = \begin{bmatrix}
A & B \\ B^{\ol{T}} & C
\end{bmatrix}
\end{equation*}
Suppose that $A$ is positive definite.
Then $M$ is positive (semi-)definite if and only if $C - B^{\ol{T}} A^{-1} B$ is positive (semi-)definite.
\end{thm}

Let $X$ be an equiangular set with angle $\alpha = \frac{1}{(2n+1)}$ and base size $K = K(X) = \frac{1+3\alpha}{2\alpha} = n+2$ in $\mathbb R^r$.
The reason for this particular combination of $\alpha$ and $K$ will be clear soon.
Let $P = \{ p_1, \dots, p_K \}$ be a $K$-base of $X$, $\Gamma$ be the subspace spanned by $P$, and $\Gamma^\perp$ be the orthogonal complement of $\Gamma$ in $\mathbb R^r$.
For the vectors $x_1, x_2 \in X \setminus P$ belonging to the same $(K,1)$ pillar,
let $x_1 = h + c_1$, $x_2 = h + c_2$ be their pillar decomposition,
that is, $h \in \Gamma$, and $c_1, c_2 \in \Gamma^\perp$.
As $h$ is a linear combination of $p_1, \dots, p_K$, we can write $h = \sum_{i=1}^K c_i p_i$ for some unknown coefficients $c_1, \dots, c_K$.  Since $x_1$ belongs to a $(K,1)$ pillar, there is an index $k_0 \in \{1,\dots, K\}$ such that
\begin{equation}
\label{eq:k1-inner}
\langle x, p_k \rangle = \langle h, p_k \rangle = \left\{
\begin{array}{ll}
\alpha, & \text{if } k = k_0; \\
-\alpha, & \text{if } k \neq k_0.
\end{array}
\right.
\end{equation}
Rewriting (\ref{eq:k1-inner}) as a matrix equation, we see that
\begin{equation}
\label{eq:k1-inner-matrix-eq}
G \cdot \begin{bmatrix}
c_1 \\ \vdots \\ c_K 
\end{bmatrix} = \alpha \cdot \bigl( 2 e_{k_0} -\sum_{i=1}^K e_i  \bigr),
\end{equation}
where $G = G(P) = (1 + \alpha) I - \alpha J$ is the Gram matrix for $P$, and $\{ e_1, \dots, e_K \}$ is the standard orthonormal basis for $\mathbb R^K$.
Since $G$ is positive and invertible, we compute
\begin{equation*}
G^{-1} = \frac{1}{1 + \alpha} I + \frac{\alpha}{(1+\alpha)(1+\alpha-K\alpha)} J.
\end{equation*}
Hence by (\ref{eq:k1-inner-matrix-eq}) we obtain that
\begin{equation*}
c_k = \left\{
\begin{array}{ll}
0,           & \text{if } k = k_0, \\
-(K-1)^{-1}, & \text{if } k \neq k_0;
\end{array}
\right.
\end{equation*}
that is,
\begin{equation*}
h = \frac{-1}{K-1} \bigl( \sum_{i=1}^K p_i - p_{k_0} \bigr).
\end{equation*}
From this expression we conclude that $\langle h, h \rangle = \alpha$.
Since $\langle x_1, x_2 \rangle = \alpha$ by Proposition~\ref{prop:K1-indep}, we conclude that $\langle \hat{c}_1, \hat{c}_2 \rangle = 0$, that is, the $c$-vectors within a single $(K,1)$ pillar are orthogonal.
(The orthogonality condition among the $c$-vectors does not hold for any other combinations of $\alpha$ and $K$.)

\begin{thm}
\label{thm:2vec}
Let $n$ be a positive integer with $n \geq 2$, and $\alpha = \frac{1}{(2n+1)}$.
Let $X$ be an equiangular set with angle $\alpha$ and base size $K = n+2$ in $\mathbb R^r$,
and we fix a base $P = \{ p_1, \dots, p_K \}$ for $X$.
If there is a $(K,1)$ pillar with at least two vectors,
then for any other $(K,1)$ pillar $\bar{x}$,
\begin{equation*}
|\bar{x}| \leq \left\{
\begin{array}{ll}
2n^2 (n+1), & \text{if $n \leq 3$}; \\
\frac12 n^2 (n+1)^2, & \text{if $n \geq 3$}.
\end{array}
\right.
\end{equation*}
\end{thm}

\begin{proof}
Let us look at the situation where two vectors come from different pillars.
Suppose that $x = h_1 + c_1$ and $u = h_2 + c_2$ in $X$ belong to distinct $(K,1)$ pillars.
Because the Hamming distance of $\varepsilon(x)$ and $\varepsilon(u)$ is $2$, we have
\begin{equation*}
\langle h_1, h_2 \rangle = \frac{n-1}{(n+1)(2n+1)}.
\end{equation*}
Therefore
\begin{equation*}
\langle \hat{c}_1, \hat{c}_2 \rangle = \frac{\langle x, u \rangle - \langle h_1, h_2 \rangle}{ \| c \|^2 }
= \frac{ \pm \frac{1}{2n+1} - \frac{n-1}{(n+1)(2n+1)} }{1 - \frac{1}{2n+1}} 
= \frac1{n(n+1)}, -\frac{1}{n+1}.
\end{equation*}

Now suppose that the pillar $\bar{u}$ contains two vectors $u_1, u_2$, and $\bar{x}$ contains $N$ vectors $x_1, \dots, x_N$.
Let $x_i = h_1 + c_i$ and $u_i = h_2 + d_i$ be their pillar decomposition.
Then the Gram matrix of $\{ \hat{c}_1, \dots, \hat{c}_N, \hat{d}_1, \hat{d}_2 \}$ has the following form:
\begin{equation*}
G = G( \hat{c}_1, \dots, \hat{c}_N, \hat{d}_1, \hat{d}_2 ) =
\begin{bmatrix}
\\
& I_N & & v_1 & v_2 \\
\\
& v_1^{\ol{T}} & & 1 & 0 \\
& v_2^{\ol{T}} & & 0 & 1
\end{bmatrix},
\end{equation*}
where $v_1$ and $v_2$ are vectors in $\mathbb R^N$ with entries in $\{ \frac{1}{n(n+1)}, \frac{-1}{n+1} \}$.
Let us assume that in $\bar{x}$, 
\begin{itemize}
\item there are $\ell_{11}$ vectors $x$ such that $\langle x, u_1 \rangle = \alpha$, $\langle x, u_2 \rangle = \alpha$;
\item there are $\ell_{12}$ vectors $x$ such that $\langle x, u_1 \rangle = \alpha$, $\langle x, u_2 \rangle = -\alpha$;
\item there are $\ell_{21}$ vectors $x$ such that $\langle x, u_1 \rangle = -\alpha$, $\langle x, u_2 \rangle = \alpha$;
\item there are $\ell_{22}$ vectors $x$ such that $\langle x, u_1 \rangle = -\alpha$, $\langle x, u_2 \rangle = -\alpha$.
\end{itemize}
Certainly $\ell_{11} + \ell_{12} + \ell_{21} + \ell_{22} = N$.
It follows that
\begin{align*}
\langle v_1, v_1 \rangle &= \frac{\ell_{11} + \ell_{12}}{n^2(n+1)^2} + \frac{\ell_{21}+\ell_{22}}{(n+1)^2}; \\
\langle v_2, v_2 \rangle &= \frac{\ell_{11} + \ell_{21}}{n^2(n+1)^2} + \frac{\ell_{12}+\ell_{22}}{(n+1)^2}; \\
\langle v_1, v_2 \rangle = \langle v_2, v_1 \rangle &= \frac{\ell_{11}}{n^2(n+1)^2} - \frac{\ell_{12}+\ell_{21}}{n(n+1)^2} + \frac{\ell_{22}}{(n+1)^2}.
\end{align*}
Since the Gram matrix $G$ is positive semidefinite, the following $2 \times 2$ matrix is also positive semidefinite by Theorem~\ref{thm:schur-decomp}:
\begin{equation*}
M := \begin{bmatrix}
1 & 0 \\ 0 & 1 
\end{bmatrix} - \begin{bmatrix}
v_1^{\ol{T}} \\ v_2^{\ol{T}}
\end{bmatrix} I_N^{-1} \begin{bmatrix}
v_1 & v_2
\end{bmatrix}
= \begin{bmatrix}
1 - \langle v_1, v_1 \rangle & - \langle v_1, v_2 \rangle \\
- \langle v_2, v_1 \rangle & 1 - \langle v_2, v_2 \rangle 
\end{bmatrix} \succcurlyeq 0.
\end{equation*}
Because $M$ is symmetric, $M$ is positive semidefinite if and only if $\op{tr} M \geq 0$ and $\det M \geq 0$.
We compute
\begin{align}
\label{eq:trM}
\frac{n^2 (n+1)^2}{2} \op{tr} M &= n^2 (n+1)^2 - \bigl(\ell_{11} + \frac{n^2 + 1}{2} (\ell_{12} + \ell_{21}) + n^2 \ell_{22} \bigr); \\
n^4 (n+1)^4 \det M &= \det \bigl( n^2 (n+1)^2 M \bigr) \nonumber \\
&= n^4 (n+1)^4 - n^2 (n+1)^2 \bigl( 2 (\ell_{11} + n^2 \ell_{22}) + (n^2+1) (\ell_{12} + \ell_{21}) \bigr) \label{eq:detM} \\
&\phantom{= } + (\ell_{11} + n^2 \ell_{22} + \ell_{12} + n^2 \ell_{21})
(\ell_{11} + n^2 \ell_{22} + \ell_{21} + n^2 \ell_{12} )  \nonumber \\
&\phantom{= } - \bigl( \ell_{11} + n^2 \ell_{22} - n (\ell_{12} + \ell_{21}) \bigr)^2.  \nonumber
\end{align}
Keep in mind that we want to maximize $N = \ell_{11} + \ell_{12} + \ell_{21} + \ell_{22}$ subject to $\op{tr} M \geq 0$, $\det M \geq 0$, and the variables $\ell_{ij}$ are all non-negative integers.
If we look closely to (\ref{eq:trM}) and (\ref{eq:detM}), the terms $\ell_{11} + n^2 \ell_{22}$ always appear as a pair, and there is no other separate term for $\ell_{11}$ and $\ell_{22}$; as a result, the sum $\ell_{11} + \ell_{22}$ is maximized when $\ell_{22} = 0$.
Henceforth we let $\ell_{22} = 0$ and continue the computation from (\ref{eq:detM}):
\begin{align}
\nonumber
n^4 (n+1)^4 \det M &= n^4 (n+1)^4 - n^2 (n+1)^2 \bigl( 2 \ell_{11} + (n^2 + 1)(\ell_{12} + \ell_{21}) \bigr) \\
\nonumber 
&\phantom{= }
+ (\ell_{11} + \ell_{12} + n^2 \ell_{21})(\ell_{11} + \ell_{21} + n^2 \ell_{12}) 
- \bigl( \ell_{11} - n (\ell_{12} + \ell_{21}) \bigr)^2 \\
\label{eq:detM2}
&= 
n^4 (n+1)^4 - n^2 (n+1)^2 \bigl( 2 \ell_{11} + (n^2 + 1)(\ell_{12} + \ell_{21}) \bigr) \\
\nonumber 
&\phantom{= }
+ (n + 1)^2 \ell_{11} (\ell_{12} + \ell_{21}) + (n^2 - 1)^2 \ell_{12} \ell_{21}.
\end{align}
The expressions and (\ref{eq:trM}) and (\ref{eq:detM2}) are symmetric with respect to $\ell_{12}$ and $\ell_{21}$, and if the sum $\ell_{12} + \ell_{21}$ is fixed, (\ref{eq:detM2}) is maximized when $\ell_{12} = \ell_{21}$ by the A.M.-G.M. inequality.
So we set $s = \ell_{11}$ and $t = \ell_{12} = \ell_{21}$ and continue the computation:
\begin{align*}
n^4 (n+1)^4 \det M &= n^4 (n+1)^4 - 2n^2 (n+1)^2 (s + (n^2+1) t) + 2(n+1)^2 st + (n^2-1)^2 t^2 \\
&= (n+1)^2 (n^2 - t) \bigl( n^2 (n+1)^2 - 2s - (n-1)^2 t \bigr).
\end{align*}
Therefore the problem becomes
\begin{equation}
\label{eq:LP}
\begin{array}{rl}
\text{to maximize} & N = s + 2t \\ \\
\text{subject to} & \left\{ \begin{array}{l}
s, t \in \mathbb Z, \quad s, t \geq 0, \\
n^2 (n+1)^2 - s - (n^2 + 1) t \geq 0, \\
(n^2 - t) \bigl( n^2 (n+1)^2 - 2s - (n-1)^2 t \bigr) \geq 0.
\end{array}
\right.
\end{array}
\end{equation}

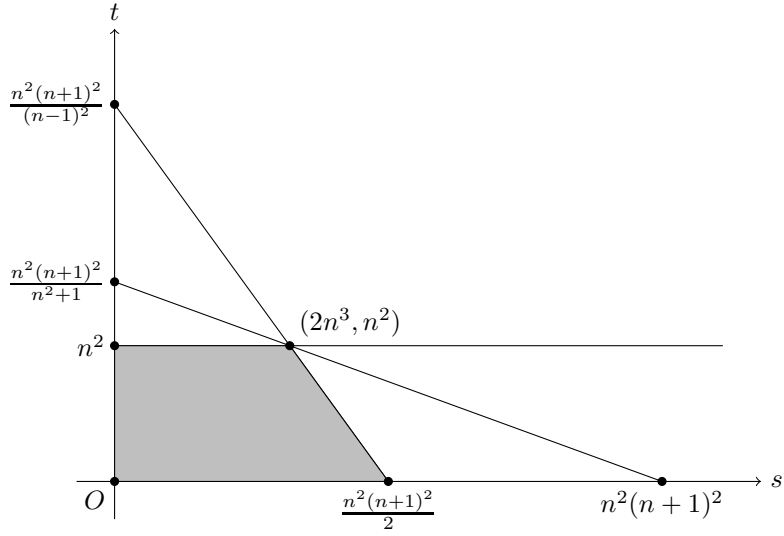
\begin{figure}
\centering
\begin{tikzpicture}
\draw [->] (-0.5, 0) -- (8.5, 0) node [right] {$s$};
\draw [->] (0, -0.5) -- (0, 6) node [above] {$t$};
\draw (0, 1.8) -- (8, 1.8);
\draw (0, 2.64705882352941) -- (7.2, 0);
\draw (0, 5) -- (3.6, 0);
\coordinate (AA) at (0, 1.8);
\coordinate (BB) at (2.304, 1.8);
\coordinate (CC) at (3.6, 0);
\draw [fill=gray!50!white] (0,0) -- (CC) -- (BB) -- (AA) -- cycle;
\draw [fill] (0,0) circle (1.5pt) node [below left ] {$O$};
\draw [fill] (CC) circle (1.5pt) node [below] {$\frac{n^2(n+1)^2}{2}$};
\draw [fill] (7.2,0) circle (1.5pt) node [below] {$n^2(n+1)^2$};
\draw [fill] (AA) circle (1.5pt) node [left] {$n^2$};
\draw [fill] (0, 2.64705882352941) circle (1.5pt) node [left] {$\frac{n^2(n+1)^2}{n^2+1}$};
\draw [fill] (0, 5)  circle (1.5pt) node [left] {$\frac{n^2(n+1)^2}{(n-1)^2}$};
\draw [fill] (BB) circle (1.5pt) node [above right] {$(2n^3, n^2)$};
\end{tikzpicture}
\caption{The feasible domain for the linear programming problem (\ref{eq:LP})}
\label{fig:LP}
\end{figure}

This is a standard problem in linear programming, 
whose feasible domain is shaded in Figure~\ref{fig:LP}.
We solve the problem and write the maximum $N_0$ of $N$ as
\begin{equation*}
N_0 = \left\{
\begin{array}{lll}
2n^2(n+1), & \text{achieved at $(s,t) = (2n^3, n^2)$}, & \text{when $n \leq 3$}; \\
\frac{1}{2} n^2 (n+1)^2, & \text{achieved at $(s,t) = (\frac{1}{2} n^2 (n+1)^2, 0)$}, & \text{when $n \geq 3$}.
\end{array}
\right.
\end{equation*}
The proof is now completed.
\end{proof}

\noindent{\bf Example.} For $n = 3$, we are looking at the angle $\alpha = \frac{1}{7}$ and the base size $K = 5$.
By Theorem~\ref{thm:2vec}, if there is a $(5,1)$ pillar with two or more vectors, then the size of another $(5,1)$ pillars is bounded by $72$.  
This maximum is achieved in two ways: the quadruple $(\ell_{11}, \ell_{12}, \ell_{21}, \ell_{22})$ defined in the proof of the theorem can be $(72, 0, 0, 0)$ or $(54, 9, 9, 0)$.

\medskip

\noindent{\bf Remark.} Following the proof of their Lemma~16, King and Tang \cite{king2016} proved that $|\bar{x}| \leq n^2 (n+1)^2$ for a $(K,1)$ pillar $\bar{x}$ if there is another nonempty $(K,1)$ pillar.
Theorem~\ref{thm:2vec} cuts their bound by half.


\section{The Lemmens-Seidel conjecture}
\label{sec:a5}

Throughout this section we assume that the common angle is $\alpha = \frac{1}{5}$.
Let us first recall a theorem in \cite{lemmens1973}.

\begin{thm}[\cite{lemmens1973}, Theorem~5.7]
  Any set of unit vectors with inner product $\pm \frac{1}{5}$ in $\mathbb R^r$, which contains $6$ unit vectors with inner product $-\frac{1}{5}$, has maximum cardinality $276$ for $23 \leq r \leq 185$, $\lfloor \frac{1}{2} (r-5) \rfloor + r + 1$ for $r \geq 185$.
\end{thm}
This theorem corresponds to the case where the common angle $\alpha = \frac{1}{5}$ and base size $K = 6$. 
Lemmens and Seidel concluded Section~5 of \cite{lemmens1973} with the following remark, which we quote here:

\smallskip

\begin{quote} {\sl
It would be interesting to know whether Theorem~5.7 holds true without the requirement of the existence of $6$ unit vectors with inner product $-\frac{1}{5}$. \dots\ The authors have obtained only partial results in this direction.
In fact, the cases where [the base size $K$] $=2, 3, 5$ have been proved, but the case [$K=4$] remains unsettled.
Yet, there is enough evidence to support the following conjecture. \dots}
\end{quote}

\smallskip

So they raised their conjecture (Conjecture~\ref{conj:LS}), but the proofs, even for the cases $K = 3, 5$, have been elusive.
Sections~3 and 4 of \cite{king2016} provided some upper bounds for $\alpha = \frac{1}{5}$.
It is well known that $|X| \leq r$ if $X \subset \mathbb R^r$ and $K = 2$~(cf.~\cite{king2016}, Corollary~2).
In this section we are going to sharpen their results and prove the conjecture when $K = 3, 5$.

\subsection{\texorpdfstring{$K = 3$}{K=3}}
\label{sec:a5-k3}

Let $X \subset \mathbb R^r$ be an equiangular set with angle $\frac{1}{5}$ in $\mathbb R^r$, with the base size $K = K(X) = 3$.
Let $P = \{ p_1, p_2, p_3 \}$ be a $3$-base in $X$, and the rest of the vectors in $X \setminus P$ are partitioned into three $(3,1)$ pillars.
By symmetry, for a unit vector $x \in X \setminus P$ that satisfies
\begin{equation*}
( \langle x, p_1 \rangle, \langle x, p_2 \rangle, \langle x, p_3 \rangle ) = \frac{1}{5} (1, -1, -1),
\end{equation*}
we can decompose $x$ into $x = h + c$, where $h \in \Gamma$ and $c \in \Gamma^\perp$.
A little computation shows that
\begin{equation*}
h = \frac{1}{9} (p_1 - 2p_2 - 2p_3).
\end{equation*}
So $\| h \|^2 = \frac{1}{9}$ and $\| c \|^2 = \frac{8}{9}$.
If $x_1 = h + c_1$ and $x_2 = h + c_2$ come from the same $(3,1)$ pillar, then $\langle x_1, x_2 \rangle = \frac{1}{5}$ by Proposition~\ref{prop:K1-indep}, henceforth $\langle \hat{c}_1, \hat{c}_2 \rangle = \frac{1}{10}$.
If $x = h_1 + c_1$ and $y = h_2 + c_2$ come from different $(3,1)$ pillars, then (by symmetry again)
\begin{equation*}
  \langle h_1, h_2 \rangle = \langle \frac{1}{9} (p_1 - 2p_2 - 2p_3), \frac{1}{9} (-2p_1 + p_2 - 2p_3) \rangle = -\frac{1}{45}.
\end{equation*}
Since
\begin{equation*}
\pm \frac{1}{5} = \langle x, y \rangle = \langle h_1, h_2 \rangle + \langle c_1, c_2 \rangle,
\end{equation*}
hence $\langle \hat{c}_1, \hat{c}_2 \rangle \in \{ \frac{1}{4}, -\frac{1}{5} \}$.

\begin{lem}
\label{lem:two-X31}
Suppose that there are two nonempty $(3,1)$ pillars.
If one of them has $4$ vectors, then the other has at most $54$ vectors.
\end{lem}

\begin{proof}
Let 
\begin{align*}
\bar{x} &= \{ h_1 + c_i \colon h_1 \in \Gamma, c_i \in \Gamma^\perp, i = 1, \dots, n \}, \\
\bar{u} &= \{ h_2 + d_i \colon h_2 \in \Gamma, d_i \in \Gamma^\perp, i = 1, 2, 3, 4 \},
\end{align*}
be two nonempty $(3,1)$ pillars.
Then the Gram matrix of $\hat{c}_i$ and $\hat{d}_i$ has the following form:
\begin{equation}
\label{eq:G-cd}
G = G( \hat{c}_1, \dots, \hat{c}_n, \hat{d}_1, \dots, \hat{d}_4 ) = 
\begin{bmatrix}
 & & & & & & & \\
 & & & & & & & \\
 & & \frac{9}{10} I_n + \frac{1}{10} J_n & v_1 & v_2 & v_3 & v_4 \\
 & & & & & & & \\
 & & v_1^{\ol{T}} & & & & & \\
 & & v_2^{\ol{T}} & \multicolumn{4}{c}{\frac{9}{10} I_4 + \frac{1}{10} J_4} \\
 & & v_3^{\ol{T}} & & & & & \\
 & & v_4^{\ol{T}} & & & & &
\end{bmatrix},
\end{equation}
where $v_1, \dots, v_4$ are column vectors whose entries are $\frac{1}{4}$ or $-\frac{1}{5}$.
Since $G$ needs to be positive semidefinite, by Theorem~\ref{thm:schur-decomp} we see that
\begin{equation}
\label{eq:a5-31-4}
M := \Bigl( \frac{9}{10} I_4 + \frac{1}{10} J_4 \Bigr) 
- V^{\ol{T}}
\Bigl( \frac{9}{10} I_n + \frac{1}{10} J_n \Bigr)^{-1} V
\succcurlyeq 0, \qquad \text{where } V := \begin{bmatrix}
v_1 & v_2 & v_3 & v_4
\end{bmatrix}.
\end{equation}
The following setup is used to facilitate the computation.
Consider the Seidel graph $S'$ generated by the vectors in $\bar{x} \cup \bar{u}$.
By Proposition~\ref{prop:K1-indep}, $S'$ is a bipartite graph because every edge must connect a vertex in $\bar{x}$ to a vertex in $\bar{u}$.
Let us classify the vectors in $\bar{x}$ by how they are connected to the vectors $u_1, \dots, u_4$ in $\bar{u}$.
Let $B_4$ be the set of binary strings of length $4$, and let $B_{4,i}$ denote the subset of $B_4$ consisting of those binary strings $b_1 b_2 b_3 b_4$ such that $\sum_j b_j = i$ for $i = 0, 1, 2, 3, 4$.
For $B = b_1 b_2 b_3 b_4 \in B_4$, let $t_B$ denote the number of vectors $h_1 + c$ in the pillar $\bar{x}$ such that
\begin{equation*}
\langle \hat{c}, \hat{d}_i \rangle = \left\{
\begin{array}{ll}
\frac{1}{4},  & \text{if $b_i = 0$}, \\
-\frac{1}{5}, & \text{if $b_i = 1$},
\end{array}
\right. \quad i = 1, 2, 3, 4.
\end{equation*}
In total there are $2^4 = 16$ variables $t_B$, $B \in B_4$, of non-negative integral values.
Obviously $n = \sum_{B \in B_4} t_B$, which is the total number of vectors in $\bar{x}$,
and $\sum_{B \in B_{4,i}} t_B$ is the number of vertices of degree $i$ in $\bar{x}$, for $i = 0, 1, 2, 3, 4$.

The vectors $v_1, v_2, v_3, v_4$ in the Gram matrix $G$ in (\ref{eq:G-cd}) has the following mutual inner products:
\begin{equation*}
\langle v_i, v_j \rangle = \frac{1}{16} \sum_{B \in B^{0,0}_{i,j}} t_B - \frac{1}{20} \sum_{B \in B^{0,1}_{i,j}} t_B + \frac{1}{25} \sum_{B \in B^{1,1}_{i,j}} t_B, \qquad i, j \in \{ 1, 2, 3, 4 \},
\end{equation*}
where $B^{k,\ell}_{i,j}$ is the subset of $B_4$ consisting of $B = b_1 b_2 b_3 b_4$ such that $\{b_i, b_j\} = \{k, \ell\}$, for $k, \ell \in \{ 0, 1 \}$.  For instance,
\begin{align*}
\inn{v}{1}{2} &= \frac{1}{16} ( t_{0000} + t_{0001} + t_{0010} + t_{0011} ) \\
&\phantom{= } - \frac{1}{20} ( t_{0100} + t_{0101} + t_{0110} + t_{0111} + t_{1000} + t_{1001} + t_{1010} + t_{1011} ) \\
&\phantom{= } + \frac{1}{25} ( t_{1100} + t_{1101} + t_{1110} + t_{1111} ).
\end{align*}
We also need
\begin{equation*}
w_i := \frac{1}{4} \sum_{\substack{ B = b_1 b_2 b_3 b_4 \in B_4 \\ b_i = 0 }} t_B - \frac{1}{5} \sum_{\substack{ B = b_1 b_2 b_3 b_4 \in B_4 \\ b_i = 1 }} t_B, \quad i \in \{ 1, 2, 3, 4 \}.  
\end{equation*}
For example,
\begin{align*}
w_1 &= \frac{1}{4} ( t_{0000} + t_{0001} + t_{0010} + t_{0011} + t_{0100} + t_{0101} + t_{0110} + t_{0111}) \\
&\phantom{= } - \frac{1}{5} ( t_{1000} + t_{1001} + t_{1010} + t_{1011} + t_{1100} + t_{1101} + t_{1110} + t_{1111}).
\end{align*}
Since
\begin{equation*}
\Bigl( \frac{9}{10} I_n + \frac{1}{10} J_n \Bigr)^{-1} = \frac{10}{9} \Bigl( I_n - \frac{1}{9+n} J_n \Bigr),
\end{equation*}
\begin{equation*}
V^{\ol{T}} I_n V
= \begin{bmatrix}
\inn{v}{i}{j}
\end{bmatrix}_{i,j=1}^4, \qquad
V^{\ol{T}} J_n V
= \begin{bmatrix}
w_i w_j
\end{bmatrix}_{i,j=1}^4.
\end{equation*}
we use these informations to expand the left-hand side of (\ref{eq:a5-31-4}) as
\begin{equation}
\label{eq:a5-31-4-a}
M = \frac{9}{10} I_4 + \frac{1}{10} J_4 - \frac{10}{9} V^{\ol{T}} I_n V + \frac{10}{9(9+n)} V^{\ol{T}} J_n V 
= \begin{bmatrix}
m_{ij}
\end{bmatrix}_{i,j=1}^4,
\end{equation}
where the entries $m_{ij}$ are
\begin{equation*}
m_{ij} = \left\{
\begin{array}{ll}
1 - \frac{10}{9} \inn{v}{i}{i} + \frac{10}{9(9+n)} w_i^2, & \text{if } i = j, \\
\frac{1}{10} - \frac{10}{9} \inn{v}{i}{j} + \frac{10}{9(9+n)} w_i w_j, & \text{if } i \neq j,
\end{array}
\qquad i, j \in \{1, 2, 3, 4 \}.
\right.
\end{equation*}
Remind that we want to maximize the sum $n = \sum_{B \in B_4} t_B$ subject to the conditions $t_B \in \mathbb Z$, $t_B \geq 0$ for all $B \in B_4$, and $M \succcurlyeq 0$.
Notice that when we set some of the variables $t_B$ to be zero, we are focusing on a particular subset of vectors in the pillar $\bar{x}$.
We argue that each of the variables $t_B$ has an upper bound as follows:
\begin{itemize}
\item Set $t_{0000} = n$ and $t_B = 0$ for all $B \neq 0000$.  Then
\begin{equation*}
M = \frac{9}{10} I_4 + \Bigl( \frac{1}{10} - \frac{5n}{8(9+n)} \Bigr) J_4.
\end{equation*}
By considering its eigenvalues, we see that $M$ is positive semidefinite if and only if 
\begin{equation*}
\frac{9}{10} + 4 \cdot \Bigl( \frac{1}{10} - \frac{5n}{8(9+n)} \Bigr) \geq 0.
\end{equation*}
Solving this inequality for $n$, we get $-9 \leq n \leq \frac{39}{4}$.
Since $n$ only assumes a non-negative integral values, we see that $0 \leq n \leq 9$;
this is the range for $t_{0000}$.

\item Set $t_{1000} = n$ and $t_B = 0$ for all $B \neq 1000$.  Then
\begin{equation*}
M = \begin{bmatrix}
1 - \frac{5n}{8(9+n)} & \frac{1}{10} - \frac{5n}{8(9+n)} & \frac{1}{10} - \frac{5n}{8(9+n)} & \frac{1}{10} + \frac{n}{2(9+n)} \\
\frac{1}{10} - \frac{5n}{8(9+n)} & 1 - \frac{5n}{8(9+n)} & \frac{1}{10} - \frac{5n}{8(9+n)} & \frac{1}{10} + \frac{n}{2(9+n)} \\
\frac{1}{10} - \frac{5n}{8(9+n)} & \frac{1}{10} - \frac{5n}{8(9+n)} & 1 - \frac{5n}{8(9+n)} & \frac{1}{10} + \frac{n}{2(9+n)} \\
\frac{1}{10} + \frac{n}{2(9+n)} & \frac{1}{10} + \frac{n}{2(9+n)} & \frac{1}{10} + \frac{n}{2(9+n)} & 1 - \frac{2n}{5(9+n)}
\end{bmatrix}
\end{equation*}
By considering non-negative values for $n$ only, our computation shows that $M$ is positive semidefinite if and only if $0 \leq n \leq 7$.  By symmetry, we conclude that $0 \leq t_B \leq 7$ for each $B \in B_{4,1}$.

\item Set $t_{1100} = n$ and $t_B = 0$ for all $B \neq 1100$.  Then
\begin{equation*}
M = \begin{bmatrix}
1 - \frac{5n}{8(9+n)} & \frac{1}{10} - \frac{5n}{8(9+n)} & \frac{1}{10} + \frac{n}{2(9+n)} & \frac{1}{10} + \frac{n}{2(9+n)} \\
\frac{1}{10} - \frac{5n}{8(9+n)} & 1 - \frac{5n}{8(9+n)} & \frac{1}{10} + \frac{n}{2(9+n)} & \frac{1}{10} + \frac{n}{2(9+n)} \\
\frac{1}{10} + \frac{n}{2(9+n)} & \frac{1}{10} + \frac{n}{2(9+n)} & 1 - \frac{2n}{5(9+n)} & \frac{1}{10} - \frac{2n}{5(9+n)} \\
\frac{1}{10} + \frac{n}{2(9+n)} & \frac{1}{10} + \frac{n}{2(9+n)} & \frac{1}{10} - \frac{2n}{5(9+n)} & 1 - \frac{2n}{5(9+n)}
\end{bmatrix}
\end{equation*}
By considering non-negative values for $n$ only, our computation shows that $M$ is positive semidefinite if and only if $0 \leq n \leq 7$.  By symmetry, we conclude that $0 \leq t_B \leq 7$ for each $B \in B_{4,2}$.

\item Set $t_{1110} = n$ and $t_B = 0$ for all $B \neq 1110$.  Then
\begin{equation*}
M = \begin{bmatrix}
1 - \frac{5n}{8(9+n)} & \frac{1}{10} + \frac{n}{2(9+n)} & \frac{1}{10} + \frac{n}{2(9+n)} & \frac{1}{10} + \frac{n}{2(9+n)} \\
\frac{1}{10} + \frac{n}{2(9+n)} & 1 - \frac{2n}{5(9+n)} & \frac{1}{10} - \frac{2n}{5(9+n)} & \frac{1}{10} - \frac{2n}{5(9+n)} \\
\frac{1}{10} + \frac{n}{2(9+n)} & \frac{1}{10} - \frac{2n}{5(9+n)} & 1 - \frac{2n}{5(9+n)} & \frac{1}{10} -- \frac{2n}{5(9+n)} \\
\frac{1}{10} + \frac{n}{2(9+n)} & \frac{1}{10} - \frac{2n}{5(9+n)}& \frac{1}{10} - \frac{2n}{5(9+n)} & 1 - \frac{2n}{5(9+n)}
\end{bmatrix}
\end{equation*}
By considering non-negative values for $n$ only, our computation shows that $M$ is positive semidefinite if and only if $0 \leq n \leq 9$.  By symmetry, we conclude that $0 \leq t_B \leq 9$ for each $B \in B_{4,3}$.

\item Set $t_{1111} = n$ and $t_B = 0$ for all $B \neq 1111$.  Then
\begin{equation*}
M = \frac{9}{10} I_4 + \Bigl( \frac{1}{10} - \frac{2n}{5(9+n)} \Bigr) J_4.
\end{equation*}
Hence $M$ is positive semidefinite if and only if $0 \leq n \leq 39$; this is the range for $t_{1111}$.
\end{itemize}

Up to this point, we find that there are only a finite number of combinations of $16$-tuples $(t_B : B \in B_4)$ that will make the matrix $M$ positive semidefinite; so far there are $10 \cdot 8^4 \cdot 8^6 \cdot 10^4 \cdot 40 \approx 2.8 \times 10^{15}$ cases to check.
To further reduce the computations, we have observed the following\footnote{The {\tt SAGE} script for this part of computations can be downloaded at
\url{http://math.ntnu.edu.tw/~yclin/two-31-pillars.sage}.}:
\begin{enumerate}[(i)]
\item Let us consider the upper bounds on the number of vertices in $\bar{x}$ of each of the degrees in the Seidel graph $S'$ (generated by $\bar{x} \cup \bar{u}$), that is, upper bounds for $\sum_{B \in B_{4,i}} t_B$, $i = 0, 1, 2, 3, 4$.
For example, when we only look for vertices of degree $1$, we set $t_B = 0$ whenever $B \in B \setminus B_{4,1}$.
Since $0 \leq t_B \leq 7$ for $B \in B_{4,1}$, we only need to pick out those quadruples $(t_{0001}, t_{0010}, t_{0100}, t_{1000}) \in \{ 0, 1, \dots, 7 \}^4$ such that the resulting matrix $M$ in (\ref{eq:a5-31-4-a}) is positive semidefinite (there are only $(7+1)^4 = 4096$ cases to check).
Among those quadruples which survive the test, the maximum for the sum $\sum_{B \in B_{4,1}} t_B$ is $16$, which occurs at $t_B = 4$ for each $B \in B_{4,1}$.

The computations for other degrees are similar and we find that
\begin{equation*}
\sum_{B \in B_{4,1}} t_B \leq 16, \qquad
\sum_{B \in B_{4,2}} t_B \leq 13, \qquad \text{and} \quad
\sum_{B \in B_{4,3}} t_B \leq 16.
\end{equation*}
This is not good enough to beat the Lemmens-Seidel 
bound\footnote{When there are two $(3,1)$ pillars with $4$ or more vectors, our computations shows that the size of whole equiangular set is bounded by $3 + 93 \cdot 3 = 282$ (see also the comparison done in Theorem~\ref{thm:a5-k3}).
But this is not enough to beat the Lemmens-Seidel's bound of $276$.}, so we proceed further.

\item We fix the value of the variable $t_{1111}$ in the range $0 \leq t_{1111} \leq 39$, and consider the maximum possible value for another variable $t_B$ for $B \in B \setminus B_{4,4}$ subject to that the matrix $M$ in (\ref{eq:a5-31-4-a}) is positive semidefinite.
To do this, we set $t_{B'} = 0$ whenever $B' \neq 1111, B' \neq B$.
Table~\ref{tbl:tb-upper-bound} lists the upper bounds for $t_B$, $B \in B_{4,i}$, $i = 0, 1, 2, 3$, when the value of $t_{1111}$ is specified.

\begin{table}
\centering
\caption{Upper bounds for $t_B$ for specified values of $t_{1111}$}
\label{tbl:tb-upper-bound}
\begin{tabular}{c|cccc|c}
& \multicolumn{4}{c}{Upper bounds for $t_B$} & \\
$t_{1111}$ & $B \in B_{4,0}$ & $B \in B_{4,1}$ & $B \in B_{4,2}$ & $B \in B_{4,3}$ & $M_{\bar{x}}$ \\ \hline
0 & 9 & 7 & 7 & 9 & 54 \\
1 & 5 & 5 & 6 & 9 & 51 \\
2 & 3 & 4 & 5 & 9 & 50 \\
3 & 2 & 3 & 5 & 9 & 46 \\
4 & 2 & 3 & 4 & 8 & 47 \\
5 & 1 & 2 & 4 & 8 & 43 \\
6 & 1 & 2 & 3 & 8 & 44 \\
7 & 1 & 2 & 3 & 8 & 45 \\
8 & 1 & 1 & 3 & 7 & 42 \\
9 & 0 & 1 & 3 & 7 & 42 \\
10, 11 & 0 & 1 & 2 & 7 & 42, 43 \\
12, 13 & 0 & 1 & 2 & 6 & 44, 45 \\
14 & 0 & 0 & 2 & 6 & 42 \\
15 & 0 & 0 & 1 & 6 & 37 \\
16--19 & 0 & 0 & 1 & 5 & 38--41 \\
20--22 & 0 & 0 & 1 & 4 & 42--44 \\
23 & 0 & 0 & 0 & 4 & 39 \\
24--27 & 0 & 0 & 0 & 3 & 36--39 \\
28--31 & 0 & 0 & 0 & 2 & 36--39 \\
32--35 & 0 & 0 & 0 & 1 & 36--39 \\
36--39 & 0 & 0 & 0 & 0 & 36--39
\end{tabular}
\end{table}
\end{enumerate}

Denote the upper bound for $t_B$ for $B \in B_{4,i}$ found in Table~\ref{tbl:tb-upper-bound} by $m_i$, $i = 0, 1, 2, 3$.
Since $|B_{4,0}| = 1$, $|B_{4,1}| = 4$, $|B_{4,2}| = 6$, and $|B_{4,3}| = 4$, an upper bound for the size of the pillar $\bar{x}$ is given by
\begin{equation*}
M_{\bar{x}} = m_0 + \min \{ 4m_1, 16 \} + \min \{ 6m_2, 13 \} + \min \{ 4m_3, 16 \} + t_{1111}.
\end{equation*}
The values for $M_{\bar{x}}$ are also listed in Table~\ref{tbl:tb-upper-bound}.
From here we conclude that the size of a $(3,1)$ pillar cannot exceed 54 when another $(3,1)$ pillar with $4$ or more vectors is present.
\end{proof}

\noindent {\bf Remark.} We note here that when a $(3,1)$ pillar $\bar{u}$ has $3$ vectors only, it is possible to have another $(3,1)$ pillar $\bar{x}$ with as many vectors as possible.  This occurs when the inner product between any one vector in $\bar{x}$ and any one vector in $\bar{u}$ is $-\frac{1}{5}$.  
Assume that $|\bar{x}| = n$.
Then the Gram matrix $G = G(\hat{c}_1, \dots, \hat{c}_n, \hat{d}_1, \hat{d}_2, \hat{d}_3)$ is
\begin{equation*}
G = \begin{bmatrix}
& & & & & \\
& \frac{9}{10} I_n + \frac{1}{10} J_n & & v & v & v \\
& & & & & \\
& v^{\ol{T}} & & & & \\
& v^{\ol{T}} & & \multicolumn{3}{c}{\frac{9}{10} I_3 + \frac{1}{10} J_3} \\
& v^{\ol{T}} & & & & 
\end{bmatrix},
\end{equation*}
where $v$ is the vector $(-\frac{1}{5}, -\frac{1}{5}, \dots, -\frac{1}{5})$ in $\mathbb R^n$, 
and $G$ has the Schur decomposition:
\begin{equation*}
\frac{9}{10} I_3 + \frac{1}{10} J_3 - \begin{bmatrix}
v^{\ol{T}} \\ v^{\ol{T}} \\ v^{\ol{T}}
\end{bmatrix} \bigl( \frac{9}{10} I_n + \frac{1}{10} J_n \bigr)^{-1} \begin{bmatrix}
v & v & v 
\end{bmatrix} = 
\frac{9}{10} I_3 + \Bigl( \frac{1}{10} - \frac{2n}{5(9+n)} \Bigr) J_3,
 \end{equation*}
which is always positive definite for any $n \in \mathbb N$.

\begin{thm}
\label{thm:a5-k3}
Let $X$ be an equiangular set with angle $\frac{1}{5}$ and base size $K(X) = 3$ in $\mathbb R^r$.
Then
\begin{equation*}
|X| \leq \max \{ 165, r+6 \}.
\end{equation*}
\end{thm}

\begin{proof}
The equiangular set $X$ is decomposed as a disjoint union of $P = \{ p_1, p_2, p_3 \}$ and three $(3,1)$ pillars.
If there are two $(3,1)$ pillars with four or more vectors, then by Lemma~\ref{lem:two-X31} we have
\begin{equation*}
|X| = |P| + |X(3,1)| \leq 3 + 54 \cdot 3 = 165.
\end{equation*}
Otherwise there is only one big $(3,1)$ pillar and the other two pillars can have at most $3$ vectors each.
Since vectors in a single $(3,1)$ pillar is linearly independent of rank $r-3$, we see that in this case
\begin{equation*}
|X| = |P| + |X(3,1)| \leq 3 + (r-3) + 3 + 3 = r + 6.
\end{equation*}
These inequalities finish the proof of the theorem.
\end{proof}

Note that $\max \{165, r + 6 \}$ is certainly less than the bound $\max \{ 276, r + 1 + \lfloor \frac{r-5}{2} \rfloor \}$ given in the Lemmens-Seidel conjecture for every $r \geq 23$,
hence we have finished the proof when the base size $K(X) = 3$.

\subsection{\texorpdfstring{$K = 4$}{K=4}}
\label{sec:a5-k4}

King and Tang (\cite{king2016}, Lemma~16) showed that $|\bar{x}| \leq 36$ for a $(4,1)$ pillar $\bar{x}$ if there is another nonempty $(4,1)$ pillar $\bar{x}$.
We get a better upper bound for $|\bar{x}|$ for $\bar{x} \in X(4,1)$ if there is another nonempty $(4,1)$ pillar $\bar{u}$ with two or more vectors by applying Theorem~\ref{thm:2vec}.
In the situation $n = 2$, so the maximum of $|\bar{x}|$ is $2n^2 (n+1) = 24$ if there is another $(4,1)$ pillar with two or more vectors.
Hence we have the following result.
\begin{prop}
\label{prop:a5-41}
In an equiangular set $X$ with angle $\frac15$ and the base size $K(X) = 4$ in $\mathbb R^r$,
the maximum number of vectors that are contained in the four $(4,1)$ pillars is $\max \{ 96, r-1 \}$.
\end{prop}

\begin{proof}
If there are two $(4,1)$ pillars with two or more vectors,
there are at most $24 \times 4 = 96$ vectors in those pillars.
Otherwise, there can be one large pillar $\bar{x}$ together with three other pillars each of which contains at most one vector.
In the case, since the vectors in $\bar{x}$ are linearly independent in the $(r-4)$-dimensional subspace $\Gamma^\perp$, the number of vectors in these $(4,1)$ pillars is at most $(r-4) + 3 = r - 1$.
\end{proof}

\noindent
{\bf Remark.} Under computations similar to Theorem~\ref{thm:2vec}, we find that if there are two nonempty $(4,1)$ pillars, then another $(4,1)$ pillar can hold at most $25$ vectors.
Hence in the case where there is only one large pillar of size $r-4$ in Proposition~\ref{prop:a5-41},
there can only be one other nonempty $(4,1)$ pillar consisting of one vector when $r - 4 > 25$, i.e., $r \geq 30$.

For each of the three $(4,2)$ pillars, the best known bound of its cardinality is $s(r-4, \frac{1}{13}, -\frac{5}{13})$ obtained in~\cite{king2016}, which denotes the number of vectors in a $2$-distance set in $\mathbb R^{r-4}$ with angles $\frac{1}{13}$ and $-\frac{5}{13}$.
With a little improvement under Proposition~\ref{prop:a5-41}, we state the result for $K = 4$.

\begin{prop}
\label{prop:a5-412}
Let $X$ be an equiangular set with the angle $\frac{1}{5}$ and base size $4$ in $\mathbb R^r$.
Then
\begin{equation}
\label{eq:a5-k4-100}
|X| \leq 100 + 3 \cdot s\bigl(r-4, \frac{1}{13}, -\frac{5}{13} \bigr).
\end{equation}
\end{prop}

\begin{proof}
The equiangular set $X$ can be partitioned into the following pairwise disjoint subsets: the $4$-base $P$, four $(4,1)$ pillars, and three $(4,2)$ pillars.
By Lemma~16 of~\cite{king2016}, any $(4,1)$ pillar $\bar{x}$ will satisfy $| \bar{x} | \leq 39$ if there is a nonempty $(4,2)$ pillar.  Since $s(r-4, \frac{1}{13}, -\frac{5}{13}) \geq r - 4$ (which can be realized if all vectors within a $(4,2)$ pillar are linearly independent), we see that
\begin{align*}
|X| &\leq |P| + |X(4,1)| + |X(4,2)| \leq 4 + 4 \cdot 24 + 3 \cdot s(r-4, \frac{1}{13}, -\frac{5}{13}) \\
&= 100 + 3 \cdot s(r-4, \frac{1}{13}, -\frac{5}{13}).
\end{align*}
\end{proof}

Notice that the right-hand side of (\ref{eq:a5-k4-100}) will never beat the Lemmens-Seidel bound.
Details will be elaborated in Section~\ref{sec:closing}.

\subsection{\texorpdfstring{$K = 5$}{K=5}}
\label{sec:a5-k5}

Let $X \subset \mathbb R^r$ be an equiangular set with angle $\frac{1}{5}$ in $\mathbb R^r$, with the base size $K = K(X) = 5$.  
Let $P = \{ p_1, p_2, p_3, p_4, p_5 \}$ be a $5$-base in $X$.
With respect to $P$, $X \setminus P$ can be partitioned into $5$ possible $(5,1)$ pillars and $10$ possible $(5,2)$ pillars.
By carefully analyzing those pillars, we answer affirmatively to the Lemmens-Seidel conjecture for the case $K = 5$.

\begin{thm}
\label{thm:a5-k5}
Let $X$ be an equiangular set with angle $\frac{1}{5}$ and base size $K(X) = 5$ in $\mathbb R^r$.
\begin{enumerate}[$(1)$]
\item If there are two or more nonempty $(5,2)$ pillars, then $|X| \leq 272$.
\item If there is at most one nonempty $(5,2)$ pillar, then $|X| \leq \frac{4}{3} r + 12$.
\end{enumerate}
\end{thm}

\begin{proof}
By Lemma 18 of~\cite{king2016}, we know that $|X(5,1)| \leq 15$.
Let us now consider the rest of the vectors $P \cup X(5,2)$.
Note that $Y := P \cup \{ p_6 \} \cup X(5,2)$ is still an equiangular set with $K(Y) = 6$ in $\mathbb R^r$,
where $p_6 = - \sum_{i=1}^5 p_i$.
Those $(5,2)$ pillars in $X$ will become $(6,3)$ pillars in $Y$, and their classifications have been discussed thoroughly by Lemmens and Seidel~\cite{lemmens1973}.
Let us recall a key fact found in the proof of Theorem~5.7 of~\cite{lemmens1973}.

\begin{lem}[\cite{lemmens1973}]
	\label{lem:LS-5.7}
	Let $Y$ be an equiangular set with angle $\frac{1}{5}$ and base size $K(Y) = 6$.
    Let $P_Y$ be a $6$-base in $Y$ and $Y$ be decomposed into $P_Y$ and various $(6,3)$ pillars.
    Suppose there are at least two nonempty pillars in $Y$.
    \begin{enumerate}[(i)]
    \item If there are two distinct pillars each of which contains a pair of adjacent vertices, then $|Y| \leq 276$.
    \item If there is only one pillar containing a pair of adjacent vertices and all other pillars contain independent vertices only, then $|Y| \leq 222$.
    \item If each of these nonempty pillars contains independent vertices only, then $|Y| \leq 258$.
    \end{enumerate}
\end{lem}

If there are two or more nonempty $(6,3)$ pillars and $|Y| > 258$,
then $Y$ must be a subset of the equiangular set $Z$ with $276$ lines in $\mathbb R^{23}$ with a $6$-base $P \cup \{ p_6 \}$ by Lemma~\ref{lem:LS-5.7}.
Goethals and Seidel~\cite{goethals1975} proved that the structure of these $276$ equiangular lines is unique, i.e., there is only one such switching class.
Here we need an explicit description of these lines.
The following detailed information can be found in~\cite{taylor1971,neumaier1984some}.
Let $\mathfrak W$ be the collection of $759$ $8$-subsets of $\{1, 2, \dots, 24 \}$ that comes from the Steiner triple system $S(5,8,24)$ (or the Witt design~\cite{witt1937}), and $\mathfrak W_1$ be the subcollection of $\mathfrak W$ consisting of those $253$ $8$-subsets that contains $1$\footnote{The complete list of these 253 $8$-subsets of $[24]$ can be found at \url{http://math.ntnu.edu.tw/\~yclin/253-8.txt}.}. 
For any $\sigma \in \mathfrak W_1$, define $w_\sigma$ be the vector in $\mathbb R^{24}$:
\begin{equation*}
w_\sigma := 4 \sum_{i \in \sigma} e_i - 4 e_1 - \sum_{j=1}^{24} e_j.
\end{equation*}
For each $k \in \{ 2, 3, \dots, 24 \}$, let $v_k := 4 e_1 + 8 e_k - \sum_{j=1}^{24} e_j$
(with $e_1, \dots, e_{24}$ being the standard basis for $\mathbb R^{24}$).
Thus
\begin{equation*}
Z_0 := \{ w_\sigma \colon \sigma \in \mathfrak W_1 \} \cup \{ v_k \colon k = 2, 3, \dots, 24 \}
\end{equation*}
gives rise to the $276$ equiangular set with angle $\frac{1}{5}$.
Note that all these $276$ vectors lie in the hyperplane $5 x_1 + \sum_{j=2}^{24} x_j = 0$,
and it is easy to see that $v_2, \dots, v_{24}$ are linearly independent,
so the span of $Z_0$ is of dimension $23$.
Consider the following $6$ elements from $\mathfrak W_1$:
\begin{equation*}
\begin{array}{ll}
\sigma_1 = \{ 1, 2, 5, 8, 13, 15, 18, 20 \}, & \sigma_4 = \{ 1, 2, 5, 8, 9, 11, 22, 24 \}, \\
\sigma_2 = \{ 1, 2, 3, 4, 9, 10, 11, 12 \}, &  \sigma_5 = \{ 1, 2, 3, 4, 17, 18, 19, 20 \}, \\
\sigma_3 = \{ 1, 3, 5, 7, 17, 19, 22, 24 \}, & \sigma_6 = \{ 1, 3, 5, 7, 10, 12, 13, 15 \}.
\end{array}
\end{equation*}
and define
\begin{equation*}
p_i = \left\{
\begin{array}{ll}
\widehat{w}_{\sigma_i},  & \text{if $i = 1, 2, 3$,} \\
-\widehat{w}_{\sigma_i}, & \text{if $i = 4, 5, 6$.}
\end{array}
\right.
\end{equation*}
Then the unit vectors $p_1, \dots, p_6$ have mutual inner products $-\frac{1}{5}$.
For the remaining $270$ vectors from $Z_0 \setminus \{ \pm p_i \colon i = 1, \dots, 6 \}$, we normalize them and pick a suitable direction for each vector so that the resulting unit vectors all have inner products $\frac{1}{5}$ with $p_6$.
Then these vectors have a pillar decomposition

\begin{equation*}
Z = \{ p_1, \dots, p_6 \} \cup \bigcup_{i=1}^{10} \bar{z}_i,
\end{equation*}
where each $\bar{z}_i$ is a $(6,3)$ pillar consisting of $27$ unit vectors, whose Seidel graph is a disjoint union of nine $3$-cliques.
So there are $90$ $3$-cliques upstairs in the pillars.
By uniqueness, we can assume that the set $Y$ above is a subset of $Z$ which contains the base set $p_1, \dots, p_6$.
If $|Y| > 258$, then $Y$ misses at most $17$ vectors in the pillars upstairs, therefore $Y$ must
contain at least one of those $90$ $3$-cliques.
Such a $3$-clique, together with the $3$ vectors in the base $p_1, \dots, p_5$ to which all vertices of this $3$-clique connect, will form a $6$-clique in $X$,
that is, $K(X) = 6$, which contradicts to the definition $K(X) = 5$.
Therefore if there are two nonempty $(5,2)$ pillars, then $|X| = |Y| - 1 + |X(5,1)| \leq 258 - 1 + 15 = 272$.  This finishes the first part of the proof.

Now let us assume that there is exactly one nonempty $(5,2)$ pillar $\bar{x}$.
There is an upper bound for the size of $|\bar{x}|$ in terms of the rank of $\bar{x}$.

\begin{lem}
\label{lem:a5k5-one}
Let $\bar{x}$ be a $(5,2)$ pillar.
Then $|\bar{x}| \leq \frac{4}{3} (d - 1)$, where $d = \op{rank}(\bar{x})$.
\end{lem}

\begin{proof}
Let $S$ be the Seidel graph of $\bar{x}$.
We first claim that $S$ does not contain any $3$-clique.
Let $P$ be the base set of $X$.
By the definition of $(5,2)$ pillars, there are three vectors in $P$ such that all vectors in $\bar{x}$ are independent to them in the Seidel graph $S(X)$ of $X$.
If there is a $3$-clique in $S$, then together with those three vectors in $P$,
they would form a complete bipartite graph $K_{3,3}$ in $S(X)$,
which is switching equivalent to $K_6$.
This contradicts to the assumption that the base size of $X$ equals $5$.

We consider $Y = P \cup \{ p_6 \} \cup \bar{x}$ again.
Now $\bar{x}$ becomes a $(6,3)$ pillar in $Y$.
By Theorem~5.1 of~\cite{lemmens1973}, any connected component of the Seidel graph $S$ of $\bar{x}$ is a subgraph of one of the graphs in Figure~\ref{fig:graphtypes}, which are those connected graphs with maximum eigenvalue $2$.
Let $S_1, \dots, S_k$ be the connected components of $S$, and $A_i$ be the adjacency matrix of $S_i$ for $i = 1, 2, \dots, k$.
The Gram matrix $G$ of $\bar{x}$ assumes the following form:
\begin{equation*}
G = \frac{1}{5} J + \frac{4}{5} I - \frac{2}{5} A, \qquad 
\text{where } A = A_1 \oplus A_2 \oplus \cdots \oplus A_k.
\end{equation*}

\fivegraphs

Let us now investigate the nullity of $G$.
We have
\begin{equation}
\label{eq:decomp-G}
x^{\ol{T}} G x = \frac{1}{5} x^{\ol{T}} J x + \frac{2}{5} x^{\ol{T}} (2I - A) x.
\end{equation}
The all-one matrix $J$ is already positive semidefinite;
so is $2I - A$, because the maximum eigenvalue of $A$ is at most $2$.
If $G x = 0$, then $x^{\ol{T}} G x = 0$;
using (\ref{eq:decomp-G}) we see that $x^{\ol{T}} J x = 0 = x^{\ol{{T}}} (2I - A) x$ as well.
The equation $x^{\ol{T}} J x = 0$ implies that the sum of the coordinates of $x$ vanishes.
As $A$ is the direct sum of the $A_i$'s, we will investigate $2I - A_i$ separately.
If $S_i$ is a proper subgraph of the five graphs listed in Figure~\ref{fig:graphtypes},
then the largest eigenvalue of $A_i$ is strictly less than $2$ by strict monotonicity~\cite{smith1970some}, and hence $2I - A_i$ is positive definite.
Therefore $v^{\ol{T}} (2I - A_i) v = 0$ implies that $v = 0$.
After relabeling, assume that $S_1$, \dots, $S_\ell$ are the components among $S_1$, \dots, $S_k$ that are listed in Figure~\ref{fig:graphtypes}.
For each $A_j$, $1 \leq j \leq \ell$, let $v_j$ be the (unique) unit eigenvector of $A_j$ with eigenvalue $2$ whose coordinates are all positive.
Then that $x^{\ol{T}} (2I - A) x = 0$ implies that $x$ lies in the span of $\tilde{v}_1$, \dots, $\tilde{v}_\ell$, where $\tilde{v}_j$ is the image of $v_j$ under the embedding induced by $A_j \rightarrow A_1 \oplus \cdots \oplus A_k$, $j = 1, \dots, \ell$.
The vectors $\tilde{v}_1$, \dots, $\tilde{v}_\ell$ are clearly linearly independent, and since their coordinates are all nonnegative, we conclude that after intersecting with the subspace $x^{\ol{T}} J x = 0$, the nullity of $G$ equals $\ell - 1$.

Except for $K_3 = C_3$, the graphs listed in Figure~\ref{fig:graphtypes} have at least $4$ vertices.
Therefore $4\ell \leq m$, where $m = |\bar{x}|$.
Denoting the rank of $\bar{x}$ by $d$, we have
\begin{equation*}
m - d = \op{null}(G) = \ell - 1 \leq \frac{m}{4} - 1 \qquad \Rightarrow \qquad
m \leq \frac{4}{3} (d-1),
\end{equation*}
which is exactly what we want to show.
\end{proof}

Back to the proof of Theorem~\ref{thm:a5-k5}.
Suppose $\bar{x}$ is the unique nonempty $(5,2)$ pillar in $X \subset \mathbb R^r$.
Since $\dim \bar{x} = \dim X - 5 = r - 5$,
we have $|\bar{x}| \leq \frac{4}{3} (r - 5 - 1) = \frac{4}{3} (r - 6)$ by Lemma~\ref{lem:a5k5-one}.
Together with $P$ and the vectors in $(5,1)$ pillars, we find
\begin{equation*}
  |X| = |P| + |X(5,1)| + |X(5,2)| \leq 5 + 15 + \frac{4}{3} (r - 6) = \frac{4}{3} r + 12,
\end{equation*}
and the proof is now completed.
\end{proof}

It is easily verified that $\max \{ 272, \frac{4}{3} r + 12 \} \leq \max \{ 276, r + 1 + \lfloor \frac{r-5}{2} \rfloor \}$ for any $r \in \mathbb N$.
Hence we have proven the Lemmens-Seidel conjecture for the case $K = 5$.


\section{Maximum equiangular sets of certain ranks}
\label{sec:max-rank}

Besides the maximum cardinality of equiangular sets in $\mathbb R^r$,
Glazyrin and Yu considered a similar question in~\cite{glazyrin2018upper}.
\begin{defn}
	Let $r$ be a positive integer. 
    We define the number $M^*(r)$ to be the maximum cardinality of equiangular lines of rank $r$.
\end{defn}

For example, we know that maximum size of equiangular line in $\mathbb{R}^8$ is 28. However, such $28$ equiangular lines in $\mathbb R^8$ actually live in a $7$-dimensional subspace by the Theorem 4 in~\cite{glazyrin2018upper}, yet $M^*(8)$ is unknown.
It is well known that $M^*(7) = 28$ and $M^*(23) = 276$.
It seems that $M^*(r)$ is an increasing function on $n$, but
Glazyrin and Yu~\cite{glazyrin2018upper} refuted this by showing $M^*(24) < 276 = M^*(23)$.
Moreover, not every value of $M^*(r)$ is known even for small $r$ in the literature, for instance $M^*(8)$.

We first deal with $M^*(8)$ and start with the following result.
The main technique of identifying saturated equiangular sets can be found in the authors' previous work~\cite{lin2018saturated}.

\begin{prop}
\label{prop:8-14}
	There are at most $14$ equiangular lines of angle $\frac13$ of rank $8$.
\end{prop}

\begin{proof}
We first construct $8\times 8$ symmetric matrices whose diagonals are $1$, and $\pm \frac{1}{3}$ elsewhere.
By considering their switching classes, we may assume that the entries in the first column and the first rows are all $\frac13$, except that the top-left corner being $1$. 
Since these matrices are Gram matrices for some bases for $\mathbb R^8$, they are required to be positive definite.
The associated graph of such a matrix is a disjoint union of a graph of $7$ vertices and one isolated vertex.
By checking all $1044$ such graphs (see \cite{flajolet2009}, Example~II.5), we find that there are only $3$ graphs that satisfy all conditions listed above. 
For each of those $3$ graphs, we collect all the unit vectors whose mutual inner products with each vector represented by the graph are $\pm \frac{1}{3}$, and transform these vectors as vertices of a new graph in which two vectors are adjacent if and only if their mutual inner products are $\pm \frac{1}{3}$.
The clique number of the new graph plus $8$ will be the size of a saturated equiangular set, and we identify the maximum in these clique numbers.
Saturated equiangular sets containing these three sets of $8$ basis vectors consist of $8$, $14$, and $14$ lines respectively, from which we conclude that $M_{\frac{1}{3}}(8) = 14$.
\end{proof}

\noindent {\bf Remark.} (Uniqueness of maximum equiangular lines of angle $\frac{1}{3}$ of rank $8$) 
Among the $3$ graphs, found in the proof of Proposition~\ref{prop:8-14}, whose associated Gram matrices are positive definite, one contains $7$ independent vertices; adding another independent vertex to the other two graphs gives two graphs that are switching equivalent: $K_2$ and $6$ independent vertices; $K_{1,6}$ and one more independent vertex.
Both of them can only be added a $6$-clique to form $14$ equiangular lines in $\mathbb R^8$ of angle $\frac{1}{3}$, so this is the only isomorphism class of $14$ equiangular lines of angle $\frac{1}{3}$ of rank $8$.

\begin{table}
	\centering
    \caption{Maximum sizes of equiangular lines with specified angles for small ranks}
    \label{tb:smallMr}
    \begin{equation*}
    \begin{array}{r|cccc}
   		\text{angle } \alpha &  \frac13 & \frac15 & \frac17 & \frac{1}{\sqrt{17}} \\
		\hline
		M_\alpha(8)  & 14 & 10 &  9 \\ \hline
		M_\alpha(9)  & 16 & 12 & 10 & 18 \\ \hline
        M_\alpha(10) & 18 & 16
    \end{array}
    \end{equation*}
\end{table}

\vskip 0.1in

Lemmens and Seidel showed that $M_{\frac{1}{3}}(r) = 2r - 2$ for $r \geq 8$ (cf.~\cite{lemmens1973}, Theorem~4.5).
The same technique as in the proof of Proposition~\ref{prop:8-14} is applied to produce Table~\ref{tb:smallMr}.
We indicate that the technique in~\cite{lin2018saturated} is more powerful than semidefinite programming method in \cite{barg2014}. For instance, the semidefinite programming bound on equiangular sets with angle $\frac{1}{5}$ in $\mathbb{R}^8$ is $11.2$ and the technique in \cite{lin2018saturated} obtains the bound $10$.

Before we proceed further, we find the following generalization of the Neumann theorem (Theorem~\ref{thm:neumann}) is necessary.

\begin{thm}[Generalization of Neumann Theorem]\label{thm:general Neu}
Let $r>3$ be a positive integer.
If there are more than $2r-2$ equiangular lines with angle $\alpha$ in $\mathbb{R}^r$, then:
\begin{itemize}
\item When $r$ is even, $\dfrac{1}{\alpha}$ is an odd integer.
\item When $r$ is odd,  $\dfrac{1}{\alpha}$ is either an odd integer or $\sqrt{2r-1}$.
Moreover, any equiangular set with angle $\frac{1}{\sqrt{2r-1}}$ of size $2r-1$ in $\mathbb R^r$ is a subset of some equiangular tight frame of size $2r$ in $\mathbb R^r$. 
\end{itemize}
\end{thm}

\begin{proof}
It suffices to assume the existence of an equiangular set $X$ with angle $\alpha$ in $\mathbb R^r$ with $|X| = 2r-1$.
Consider its Seidel matrix $A=\frac{1}{\alpha} (G(X)-I)$, which
is a $(2r-1) \times (2r-1)$ symmetric matrix with integer coefficients and diagonal entries are all zeros,
but non-diagonal entries are either $1$ or $-1$. 
The matrix $A$ will have an eigenvalue $a = \frac{-1}{\alpha}$ with multiplicity at least $r-1$, since $G$ has an eigenvalue zero with multiplicity at least $r-1$.
If $a$ is rational, then $a$ must be an odd integer using the same argument as in the proof of the original Neumann theorem (cf.~\cite{lemmens1973}, Theorem~3.4).
Otherwise $a$ is irrational, and by degree count $a$ must a zero of an irreducible quadratic polynomial over $\mathbb Z$, which we write as $x^2 - c_1 x + c_2$; let $a^*$ be the other zero of this quadratic polynomial, which must also be an eigenvalue of $A$.

The characteristic polynomial of $A$ assumes the following form: $\op{char}(A) = (x^2-c_1x+c_2)^{r-1}(x-c_3)$, with $c_i$ being all integers. 
Comparing the coefficients, we get:
\begin{equation}
\label{eq:rootsum}
c_3+(r-1)c_1 = \op{tr} A = 0, \qquad \text{i.e.,} \qquad c_3 = -(r-1) c_1.
\end{equation}
Next, we see that
\begin{equation}
\label{eq:squaresum}
(r-1) (a^2 + {a^*}^2) + c_3^2 = \op{tr} A^2 = (2r-1) (2r-2).
\end{equation}
By Vieta's formula, $a^2 + {a^*}^2 = (a + a^*)^2 - 2 a a^* = c_1^2 - 2 c_2$.
Plug in this relation and (\ref{eq:rootsum}) back to (\ref{eq:squaresum}), we see that
\begin{equation*}
r c_1^2 - 2 c_2 = 4r-2, \qquad \text{i.e.,} \qquad c_2 = \frac{r}{2} (c_1^2 - 4) + 1.
\end{equation*}
Because $a$ and $a^*$ are distinct real roots of the quadratic equation $x^2 - c_1 x + c_2 = 0$, its discriminant must be positive, that is,
\begin{equation*}
0 < \Delta = c_1^2 - 4c_2 = (1-2r)(c_1^2-4).
\end{equation*}
Since $r \in \mathbb N$, that $\Delta > 0$ implies that $c_1 \in \{ 1, -1, 0 \}$. 
We look into these three cases separately:
\begin{itemize}
\item $c_1 =  1$.  Then $(c_2, c_3) = (-\dfrac{3r}{2} + 1, -(r-1) )$.  However, this case is not allowed, because the matrix $G$, being a Gram matrix, must be positive semidefinite, hence $0$ is the smallest eigenvalue of $G$, and this implies that $a = -\frac{1}{\alpha} = \frac{-1-\sqrt{6r-3}}{2}$ is the smallest eigenvalue of $A$.  However, $c_3 = -(r-1)$, which is also an eigenvalue of $A$ by assumption, is always smaller than $a$, and this is a contradiction. 
\item $c_1 = -1$.  Then $(c_2, c_3) = ( -\dfrac{3r}{2} + 1, r - 1 )$.  Because $c_2 \in \mathbb Z$, we see that $r$ has to be an even integer; let $r = 2t$ for some $t \in \mathbb N$, $t > 1$.  
Now we compute the determinant of $A$ as $\det A = c_2^{r-1} \cdot c_3 = (1 - 3t)^{2t-1} \cdot (2t-1)$.  If $t$ is even, then $\det A$ is an odd integer; if $t$ is odd, then $\det A$ is a multiple of $2^5$ since $t$ is at least $3$.
Both contradict Corollary~3.6 of \cite{greaves2016} which states that $\det A \equiv 1 - (4t-1) = 2 - 4t \equiv 2 \pmod{4}$ when $A$ is a Seidel matrix of order $n = 2r-1 = 4t-1$.

\item $c_1 =  0$.  Then $(c_2, c_3) = (-2r+1, 0)$, and the angle is $\alpha = \dfrac{1}{\sqrt{2r-1}}$.
The following proof is taken from~\cite{greaves2017symmetric}.
Assume that there exists such a $(2r-1) \times (2r-1)$ symmetric matrix $A$ from the equiangular set $X$.
Then the characteristic polynomial of $A$ must be $x (x^2 - 2r + 1)^{2t-1}$.
Consider the matrix $M = (2r - 1) I - A^2$.
Then the spectrum of $M$ is $\{ [0]^{2r-2}, [2r-1]^1 \}$, which implies that $M$ is positive semidefinite and of rank $1$.
Observe that $M$ has diagonal entries only $1$ and has off-diagonal integer entries congruent to $1 \pmod{2}$.
Since $M$ is of rank $1$, $M$ is a $(1,-1)$-matrix.
By switching $M$ so that the entries in the first column and the first row are all $1$, it follows from the rank of $M$ being $1$ that $M = J$, the all-one matrix.
Multiplying the all-one column vector $\bm{1}$ of length $2r-1$ by $(2r-1) I - A^2 = J$, we have $(2r-1) \bm{1} - A^2 \bm{1} = (2r-1) \bm{1}$.
This implies that $A \bm{1} = 0$.

Now consider a symmetric matrix $\displaystyle B = \begin{bmatrix} A & \bm{1} \\ \bm{1}^{\ol{T}} & 0 \end{bmatrix}$ of order $2r$ with $0$ on the diagonal and $\pm 1$ otherwise.
Then
\begin{equation*}
  B^2 = \begin{bmatrix}
  A^2 + J & A \bm{1} \\ \bm{1}^{\ol{T}} A & \bm{1}^{\ol{T}} \bm{1}  
 \end{bmatrix}
 = \begin{bmatrix}
   (2r-1) I & 0 \\ 0 & (2r-1)
 \end{bmatrix} = (2r-1) I.
\end{equation*}
Thus $B$ is a symmetric conference matrix of order $2r$.
However, it is known in \cite[Corollary~2.2]{delsarte1971orthogonal} that the order of a symmetric conference matrix must be congruent to $2$ modulo $4$, hence $r$ must be an odd integer in this case.
It also follows that $B$ has the eigenvalues $\pm \sqrt{2r-1}$, and thus $G = I - \frac{1}{\sqrt{2r-1}} B$ is positive semidefinite.
This matrix $G$ is the Gram matrix of an $r \times 2r$ equiangular tight frame.
The upper-left principal submatrix of size $(2r-1)$ of $G$ is the Gram matrix of the equiangular set $X$.
Therefore $X$ is in some equiangular tight frame of size $2r$ in $\mathbb R^r$.
\end{itemize}
\end{proof}


\noindent{\bf Remark.}
We also consider possible irrational angles $\alpha$ that could produce $14$ equiangular lines in $\mathbb R^8$.
Let $A$ be the Seidel matrix of $14$ equiangular lines $X$ of rank $8$ with angle $\alpha \in \mathbb R \setminus \mathbb Q$.
The number $a = -1/\alpha$ is the smallest eigenvalue of $A$ of multiplicity $6$ and hence $\mathbb Q(a)$ is a quadratic number field.
We may assume that the characteristic polynomial is of the following form:
\begin{equation*}
  (x^2 - c_1 x + c_2)^6 (x^2 - c_3 x + c_4), \qquad
  c_1, c_2, c_3, c_4 \in \mathbb Z.
\end{equation*}
The trace conditions on $A$ and $A^2$ imply that
\begin{align*}
    0 &= \op{tr} A = 6 c_1 + c_3, \\
    182 = 14 \cdot 13 &= \op{tr} A^2 = 6 (c_1^2 - 2c_2) + (c_3^2 - 2 c_4).
\end{align*}
Since all eigenvalues of $A$ are real and $a$ is not rational, we also have $c_1^2 > 4 c_2$ and $c_3^2 \geq 4 c_4$.
Putting all these conditions together, we find that $a = -1/\alpha$ could only be the negative root of some quadratic equation $x^2 - c_1 x + c_2 = 0$, where 
\begin{align*}
    (c_1, c_2) \in \{ & (-2, -7), (-2, -6), (-2, -5), (-2, -4), (-2, -2), (-2, -1), (-1, -13), \\ 
    & (-1, -11), (-1, -10), (-1, -9), (-1, -8), (-1, -7), (-1, -5), (-1, -4), \\ 
    & (-1, -3), (-1, -1), (0, -15), (0, -14), (0, -13), (0, -12), (0, -11), \\
    & (0, -10), (0, -8), (0, -7), (0, -6), (0, -5), (0, -3), (0, -2), (1, -13) \\
    & (1, -11), (1, -10), (1, -9), (1, -8), (1, -7), (1, -5), (1, -4), (1, -3), \\
    & (1, -1), (2, -7), (2, -6), (2, -5), (2, -4), (2, -2), (2, -1) \}.
\end{align*}
We look into each of the possible angles to conclude that $\alpha = \frac{2\sqrt{2}-1}{7}$ is the only irrational angle with which there are $14$ equiangular lines of rank $8$.

\vskip 0.1in

We also need the inequality~(\ref{eq:relative-bound}), which is the so-called \emph{relative bound} for equiangular lines.
\begin{thm}[\cite{vanlint1966}, p.342]
\label{thm:relative-bound}
Let $X$ be an equiangular set with angle $\alpha$ in $\mathbb R^r$.
If $r < \frac{1}{\alpha^2}$, then
\begin{equation}
\label{eq:relative-bound}
|X| \leq \frac{r (1 - \alpha^2)}{1 - r \alpha^2}.
\end{equation}
\end{thm}

Together with Theorems~\ref{thm:general Neu} and~\ref{thm:relative-bound}, we realize that for each positive integer $r$ there are only a couple of angles to be checked to determine $M^*(r)$.

\begin{thm}
\label{thm:m-star-8-9-10}
We have
\begin{equation*}
M^*(8) = 14, \qquad M^*(9) = 18, \qquad \text{and} \qquad M^*(10) = 18.
\end{equation*}
\end{thm}

\begin{proof}
For the case $r=8$, 
we only need to check the $\alpha$-values for $\{\frac{1}{3}, \frac{1}{5}, \frac{1}{7} \}$ by Theorems~\ref{thm:general Neu} and \ref{thm:relative-bound}.
Since $M_{\frac13}(8) = 14$, $M_{\frac15}(8) = 10$, and $M_{\frac17}(8) = 9$, as listed in Table~\ref{tb:smallMr},
we conclude that $M^*(8) = 14$.

For $n = 9$, we read from Table~\ref{tb:smallMr} that $M_{\frac{1}{3}}(9) = 16$ and $M_{\frac{1}{5}}(9) = 12$.
By Theorem~\ref{thm:relative-bound} we obtain that $M_{\frac{1}{2n+1}}(9) \leq 10$ for positive integers $n \geq 3$.
Finally we find that $M_{\frac{1}{\sqrt{17}}}(9) = 18$, which can be constructed by the Paley graph with cardinality 17 (see \cite{waldron2009construction}).
Hence $M^*(9) = 18$.

Table~\ref{tb:smallMr} shows that $M_{\frac{1}{3}}(10) = 18$,
and $M_{\frac{1}{2n+1}}(10) \leq 16$ for every positive integer $n \geq 2$ by Theorem~\ref{thm:relative-bound}.
According to Theorem~\ref{thm:general Neu}, we do not need to check any other angles.
Hence $M^*(10) = 18$.
\end{proof}

Notice that Theorem~\ref{thm:general Neu} is universal for every dimension $r$.
We may solve for more exact values of $M^*(r)$ if we spend more time on computer calculation.
However the work will be repetitious so we stop here.


\section{Closing remarks}
\label{sec:closing}
We note that the results of Theorem~\ref{thm:2vec} and Lemma~\ref{lem:two-X31} are not optimal in the sense that the upper bound on the cardinality of a pillar can be lowered if more vectors are presents in another pillar.
For instance, with the angle $\frac{1}{5}$ and base size $4$, it should not be possible to have four $(4,1)$ pillars with $24$ vectors each (this produces the number $96$ in Proposition~\ref{prop:a5-41}).
Nevertheless our bounds are sufficient to beat Lemmens-Seidel's conjecture, so we did not pursue further.
On the other hand, these bounds are valid regardless of the dimensions or ranks where the equiangular set lives.

Based on our experiments, we believe that there can only be a large pillar; by this we form the following conjecture.

\begin{conj}
There is a constant $C$ that depends on the angle $\alpha$ and the base size $K$, but not to the dimension or rank, of any equiangular set, such that there could not be two pillars of size at least $C$.
\end{conj}

This conjecture is coherent to Sudakov's result that when the angle is fixed except for $\frac{1}{3}$, the upper bound for equiangular sets in $\mathbb R^r$ is at most $1.92 r$ asymptotically (see~\cite{balla2018equiangular}).
Sudakov had a construction of equiangular sets with angle $\alpha = \frac{1}{2n+1}$ and rank $r$ which concentrates in one pillar whose cardinality is asymptotic to $\frac{(n+1) r}{n}$ for every positive integer $n$~(see \cite{balla2018equiangular}, Conjecture~6.1).

The only unsolved case towards the Lemmens-Seidel conjecture is the $(4,2)$ pillars.
King and Tang~\cite{king2016} showed that the unit vectors within one $(4,2)$ pillar form a $2$-distance set of angles $\frac{1}{13}$ and $-\frac{5}{13}$.
But the semidefinite linear programming bound $s(r, \frac{1}{13}, -\frac{5}{13})$ cannot be small.
Consider the $3\ell \times 3\ell$ matrix of the following block form:
\begin{equation*}
\begin{bmatrix}
B & \frac{1}{13} J_3 & \cdots & \frac{1}{13} J_3 \\
\frac{1}{13} J_3 & B & \cdots & \frac{1}{13} J_3 \\
\vdots & \vdots & \ddots & \vdots \\
\frac{1}{13} J_3 & \frac{1}{13} J_3 & \cdots & B
\end{bmatrix}, \qquad 
\text{where\ } B = \begin{bmatrix}
1 & -\frac{5}{13} & -\frac{5}{13} \\
-\frac{5}{13} & 1 & - \frac{5}{13} \\
-\frac{5}{13} & -\frac{5}{13} & 1
\end{bmatrix}_{3 \times 3}
\end{equation*}
This matrix has rank $2\ell+1$ and positive semidefinite, so it is the Gram matrix of $3\ell$ vectors of rank $2\ell+1$.
On the other hand, the base size of the equiangular set generated from this matrix is $6$, for in such a pillar there are many independent $3$-cliques, and two independent $3$-cliques are switching equivalent to a $6$-clique (by switching all three vertices in one of the $3$-cliques).
So we raise another conjecture which relates to Theorem~5.1 of~\cite{lemmens1973}.
\begin{conj}
In the case where $\alpha = \frac{1}{5}$ and base size $K=4$, there are only a finite number of families of connected graphs $S_i$'s such that the connected components of the Seidel graph of any $(4,2)$ pillar in an equiangular set is either a graph or a subgraph of a graph in $S_i$. 
\end{conj}

\subsection*{Acknowledgements}
The authors thank Eiichi Bannai and Alexey Glazyrin for their helpful discussions on this work.   
This material is based upon work supported by the National Science Foundation under
Grant No.~DMS-1439786 while the second author was in residence at the Institute for 
Computational and Experimental Research in Mathematics in Providence, RI, during the Point configurations in Geometry, Physics and Computer Science Program. Part of this work was done when the second author visited National Center for Theoretical Sciences (NCTS), Taiwan, in the summer of 2018. The authors are grateful to the support of NCTS.
Finally, the authors would like to thank anonymous referees who offer useful comments on improvement of Theorem~\ref{thm:general Neu} and give clues to the remark following Theorem~\ref{thm:general Neu}. 

\nocite{sagemath}

\bibliographystyle{amsalpha}
\bibliography{equiangular}

\end{document}